\documentclass[11pt,reqno]{amsart}

\usepackage{amsmath,amsfonts,amsthm,amssymb}
\usepackage{graphicx}
\usepackage{verbatim}
\usepackage{color}
\usepackage{amscd}
\usepackage{verbatim}

\begin{document}
\newcommand{\M}{{\mathcal M}}
\newcommand{\loc}{{\mathrm{loc}}}
\newcommand{\core}{C_0^{\infty}(\Omega)}
\newcommand{\sob}{W^{1,p}(\Omega)}
\newcommand{\sobloc}{W^{1,p}_{\mathrm{loc}}(\Omega)}
\newcommand{\merhav}{{\mathcal D}^{1,p}}
\newcommand{\be}{\begin{equation}}
\newcommand{\ee}{\end{equation}}
\newcommand{\mysection}[1]{\section{#1}\setcounter{equation}{0}}
\newcommand{\laplace}{\Delta}
\newcommand{\pl}{\laplace_p}
\newcommand{\grad}{\nabla}
\newcommand{\pd}{\partial}
\newcommand{\bo}{\pd}
\newcommand{\csub}{\subset \subset}
\newcommand{\sm}{\setminus}
\newcommand{\ssm}{:}
\newcommand{\diver}{\mathrm{div}\,}
\newcommand{\bea}{\begin{eqnarray}}
\newcommand{\eea}{\end{eqnarray}}
\newcommand{\bean}{\begin{eqnarray*}}
\newcommand{\eean}{\end{eqnarray*}}
\newcommand{\thkl}{\rule[-.5mm]{.3mm}{3mm}}
\newcommand{\cw}{\stackrel{\rightharpoonup}{\rightharpoonup}}
\newcommand{\id}{\operatorname{id}}
\newcommand{\supp}{\operatorname{supp}}
\newcommand{\wlim}{\mbox{ w-lim }}
\newcommand{\mymu}{{x_N^{-p_*}}}
\newcommand{\R}{{\mathbb R}}
\newcommand{\N}{{\mathbb N}}
\newcommand{\Z}{{\mathbb Z}}
\newcommand{\Q}{{\mathbb Q}}
\newcommand{\abs}[1]{\lvert#1\rvert}
\newtheorem{theorem}{Theorem}[section]
\newtheorem{corollary}[theorem]{Corollary}
\newtheorem{lemma}[theorem]{Lemma}
\newtheorem{notation}[theorem]{Notation}
\newtheorem{definition}[theorem]{Definition}
\newtheorem{remark}[theorem]{Remark}
\newtheorem{proposition}[theorem]{Proposition}
\newtheorem{assertion}[theorem]{Assertion}
\newtheorem{problem}[theorem]{Problem}
\newtheorem{conjecture}[theorem]{Conjecture}
\newtheorem{question}[theorem]{Question}
\newtheorem{example}[theorem]{Example}
\newtheorem{Thm}[theorem]{Theorem}
\newtheorem{Lem}[theorem]{Lemma}
\newtheorem{Pro}[theorem]{Proposition}
\newtheorem{Def}[theorem]{Definition}
\newtheorem{Exa}[theorem]{Example}
\newtheorem{Exs}[theorem]{Examples}
\newtheorem{Rems}[theorem]{Remarks}
\newtheorem{Rem}[theorem]{Remark}

\newtheorem{Cor}[theorem]{Corollary}
\newtheorem{Conj}[theorem]{Conjecture}
\newtheorem{Prob}[theorem]{Problem}
\newtheorem{Ques}[theorem]{Question}
\newtheorem*{corollary*}{Corollary}
\newtheorem*{theorem*}{Theorem}
\newcommand{\pf}{\noindent \mbox{{\bf Proof}: }}


\renewcommand{\theequation}{\thesection.\arabic{equation}}
\catcode`@=11 \@addtoreset{equation}{section} \catcode`@=12
\newcommand{\Real}{\mathbb{R}}
\newcommand{\real}{\mathbb{R}}
\newcommand{\Nat}{\mathbb{N}}
\newcommand{\ZZ}{\mathbb{Z}}
\newcommand{\CC}{\mathbb{C}}
\newcommand{\Pess}{\opname{Pess}}
\newcommand{\Proof}{\mbox{\noindent {\bf Proof} \hspace{2mm}}}
\newcommand{\mbinom}[2]{\left (\!\!{\renewcommand{\arraystretch}{0.5}
\mbox{$\begin{array}[c]{c}  #1\\ #2  \end{array}$}}\!\! \right )}
\newcommand{\brang}[1]{\langle #1 \rangle}
\newcommand{\vstrut}[1]{\rule{0mm}{#1mm}}
\newcommand{\rec}[1]{\frac{1}{#1}}
\newcommand{\set}[1]{\{#1\}}
\newcommand{\dist}[2]{$\mbox{\rm dist}\,(#1,#2)$}
\newcommand{\opname}[1]{\mbox{\rm #1}\,}
\newcommand{\mb}[1]{\;\mbox{ #1 }\;}
\newcommand{\undersym}[2]
 {{\renewcommand{\arraystretch}{0.5}  \mbox{$\begin{array}[t]{c}
 #1\\ #2  \end{array}$}}}
\newlength{\wex}  \newlength{\hex}
\newcommand{\understack}[3]{%
 \settowidth{\wex}{\mbox{$#3$}} \settoheight{\hex}{\mbox{$#1$}}
 \hspace{\wex}  \raisebox{-1.2\hex}{\makebox[-\wex][c]{$#2$}}
 \makebox[\wex][c]{$#1$}   }%
\newcommand{\smit}[1]{\mbox{\small \it #1}}
\newcommand{\lgit}[1]{\mbox{\large \it #1}}
\newcommand{\scts}[1]{\scriptstyle #1}
\newcommand{\scss}[1]{\scriptscriptstyle #1}
\newcommand{\txts}[1]{\textstyle #1}
\newcommand{\dsps}[1]{\displaystyle #1}
\newcommand{\dx}{\,\mathrm{d}x}
\newcommand{\dy}{\,\mathrm{d}y}
\newcommand{\dz}{\,\mathrm{d}z}
\newcommand{\dt}{\,\mathrm{d}t}
\newcommand{\dr}{\,\mathrm{d}r}
\newcommand{\du}{\,\mathrm{d}u}
\newcommand{\dv}{\,\mathrm{d}v}
\newcommand{\dV}{\,\mathrm{d}V}
\newcommand{\ds}{\,\mathrm{d}s}
\newcommand{\dS}{\,\mathrm{d}S}
\newcommand{\dk}{\,\mathrm{d}k}

\newcommand{\dphi}{\,\mathrm{d}\phi}
\newcommand{\dtau}{\,\mathrm{d}\tau}
\newcommand{\dxi}{\,\mathrm{d}\xi}
\newcommand{\deta}{\,\mathrm{d}\eta}
\newcommand{\dsigma}{\,\mathrm{d}\sigma}
\newcommand{\dtheta}{\,\mathrm{d}\theta}
\newcommand{\dnu}{\,\mathrm{d}\nu}

\def\ga{\alpha}     \def\gb{\beta}       \def\gg{\gamma}
\def\gc{\chi}       \def\gd{\delta}      \def\ge{\epsilon}
\def\gth{\theta}                         \def\vge{\varepsilon}
\def\gf{\phi}       \def\vgf{\varphi}    \def\gh{\eta}
\def\gi{\iota}      \def\gk{\kappa}      \def\gl{\lambda}
\def\gm{\mu}        \def\gn{\nu}         \def\gp{\pi}
\def\vgp{\varpi}    \def\gr{\rho}        \def\vgr{\varrho}
\def\gs{\sigma}     \def\vgs{\varsigma}  \def\gt{\tau}
\def\gu{\upsilon}   \def\gv{\vartheta}   \def\gw{\omega}
\def\gx{\xi}        \def\gy{\psi}        \def\gz{\zeta}
\def\Gg{\Gamma}     \def\Gd{\Delta}      \def\Gf{\Phi}
\def\Gth{\Theta}
\def\Gl{\Lambda}    \def\Gs{\Sigma}      \def\Gp{\Pi}
\def\Gw{\Omega}     \def\Gx{\Xi}         \def\Gy{\Psi}

\renewcommand{\div}{\mathrm{div}}
\newcommand{\red}[1]{{\color{red} #1}}


\title{Optimal $L^p$ Hardy-type inequalities}

\author{Baptiste Devyver}
\address{Baptiste Devyver, Department of Mathematics,  Technion - Israel Institute of Technology, Haifa 32000, Israel}
\email{baptiste.devyver@univ-nantes.fr}
\author{Yehuda Pinchover}
\address{Yehuda Pinchover,
Department of Mathematics, Technion - Israel Institute of
Technology,   Haifa 32000, Israel}
\email{pincho@techunix.technion.ac.il}

\maketitle

\begin{abstract}

Let $\Gw$ be a domain in $\R^n$ or a noncompact Riemannian manifold of dimension $n\geq 2$, and $1<p<\infty$. Consider the functional $\mathcal{Q}(\varphi):=\int_\Omega \big(|\nabla \varphi|^p+V|\varphi|^p\big)\dnu$ defined on $\core$, and assume that $\mathcal{Q}\geq 0$. The aim of the paper is to generalize to the quasilinear case ($p\neq 2$) some of the results obtained in \cite{DFP} for the linear case ($p=2$), and in particular, to obtain ``as large as possible" nonnegative (optimal) Hardy-type weight $W$ satisfying
$$\mathcal{Q}(\varphi)\geq \int_{\Omega} W|\varphi|^p\dnu \quad\forall \varphi\in\core.$$

Our main results deal with the case where $V=0$, and $\Gw$ is a general punctured domain (for $V\neq0$ we obtain only some partial results). In the case $1<p\leq  n$, an optimal Hardy-weight is given by  $$W:=\left(\frac{p-1}{p}\right)^p\left|\frac{\nabla G}{G}\right|^p,$$
where $G$ is the associated positive minimal Green function with a pole at $0$. On the other hand, for $p>n$, several cases should be considered, depending on the behavior of $G$ at infinity in $\Omega$. The results are extended to annular and exterior domains.
\end{abstract}

\mysection{Introduction}
In a recent paper \cite{DFP}, the authors studied a general second-order {\em linear} elliptic operator $P\geq 0$ in a general domain $\Gw\subset \mathbb{R}^n$ (or a noncompact smooth manifold of dimension $n$), where $n\geq 2$, and obtained an optimal improvement of the inequality $P\geq 0$. The improved inequality is of the form $P\geq W$, where $W$ is ``as large as possible" weight function, and (in the self-adjoint case) the inequality $P\geq W$ is meant in the quadratic form sense. The weight $W$ is given explicitly  using a simple construction called the {\em supersolution construction}; any two linearly independent positive (super)solutions $u_0,u_1$ of the equation $Pu=0$ give rise to a one-parameter family of Hardy-type weights $\{W_\ga\}_{\{0\leq \ga \leq 1\}}$ satisfying the inequality $P\geq W_\ga$ (for more details on this construction see Section~\ref{sec33}). The optimal weight is obtained by a careful choice of $u_0,u_1$ and $\ga$.

In the case of a Schr\"odinger type operator $P$, the main result of \cite{DFP} reads as follows.

\begin{theorem}\label{temp_def}
Consider a {\em symmetric} second-order linear elliptic operator $P$ of the form
$$Pu:=-\div\big(A(x)\nabla u\big)+V(x)u$$
which is subcritical in $\Gw$. Let $q$ be the associated quadratic form. Then there exists a nonzero, nonnegative weight $W$ satisfying the following properties:
\begin{enumerate}
\item[(a)] The following Hardy-type inequality holds true
 \begin{equation}
\label{HardyType_def}
q(\vgf) \geq \lambda \int_\Omega W(x) |\varphi(x)|^2 \dx \qquad \forall\varphi \in C_0^\infty(\Omega),
\end{equation}
with $\gl>0$. Denote by $\lambda_0:=\lambda_0(P,W,\Gw)$ the best constant satisfying \eqref{HardyType_def}.

\item[(b)] The operator   $P-\lambda_0 W$ is \textit{critical} in $\Gw$; that is, the inequality
$$q(\vgf) \geq  \int_{\Gw} W_1(x)\vgf^2(x) \dx \qquad \forall \vgf \in C_0^\infty(\Omega)$$
is not valid for any $W_1 \gneqq \lambda_0W$.

\item[(c)] The constant $\lambda_0$ is also the best constant for \eqref{HardyType_def} with test functions supported in the exterior of any fixed compact set in $\Omega$.
\item[(d)] The operator $P - \lambda_0 W$ is {\it null-critical} in $\Gw$; that is, the corresponding Rayleigh-Ritz variational problem
\begin{equation}\label{RRVP}
\inf_{\vgf\in \mathcal{D}_P^{1,2}(\Gw)}\left\{\frac{q(\varphi)}{\int_{\Gw} W(x) |\vgf(x)|^2 \dx}\right\}
\end{equation}
 admits no minimizer. Here $\mathcal{D}_P^{1,2}(\Omega)$ is the completion of $C_0^\infty(\Omega)$ with respect to the norm $u\mapsto \sqrt{q(u)}$.
\item[(e)] If furthermore, $W>0$, then the spectrum and the essential spectrum of the Friedrichs extension of the operator $W^{-1}P$ on $L^2(\Gw, W\dx)$ are both equal to  $[\lambda_0,\infty)$.
\end{enumerate}
\end{theorem}

In the present paper we consider the quasilinear case. Let $1<p<\infty$, and denote by $\Gd_p(u):=\div(|\nabla u|^{p-2}\nabla u)$ the $p$-Laplace operator. Throughout the paper, $\Gw$ is either a domain in $\R^n$, or a noncompact smooth Riemannian manifold of dimension $n$, $n\geq 2$, such that $0\in \Gw$.  Let $V\in L^\infty_{\mathrm{loc}}(\Omega)$ be a real valued potential, and let $Q_V$ be the quasilinear operator
\begin{equation}\label{eqQ}
Q_V(u)=Q(u):=-\div(|\nabla u|^{p-2}\nabla u)+V(x)|u|^{p-2}u
\end{equation}
defined on $\Omega$. Denote by
$$\mathcal{Q}_V(\varphi)=\mathcal{Q}(\varphi):=\int_\Omega \big(|\nabla \varphi|^p+V|\varphi|^p\big)\dnu$$ the associated energy defined on $C^\infty_0(\Gw)$. We say that
$\mathcal{Q}\geq 0$ in $\Gw$  if $\mathcal{Q}(\varphi)\geq 0$ for all  $\varphi\in C^\infty_0(\Gw)$.

Let $W\geq 0$ in $\Gw$.  We denote
\begin{align*}
\lambda_0(Q_V,W,\Omega)&:=\sup\{\gl\in\Real\mid \mathcal{Q}_{V-\gl W}\geq 0 \; \mbox{ in }  \Gw \},\\[2mm]
\lambda_\infty(Q_V,W,\Omega)&:=\sup\{\gl\in\Real\mid \exists K\subset
\subset\Gw\mbox{ s.t. } \mathcal{Q}_{V-\gl W}\geq 0 \; \mbox{ in }  \Gw\setminus K \},
\end{align*}
respectively, the {\em best constant} and {\em best constant at infinity} in the Hardy-type inequality

$$\mathcal{Q}_V(\varphi) \geq \gl \int_{\Omega^\star} W|\varphi|^p\dnu,\qquad \forall\varphi\in C_0^\infty(\Omega).$$

The aim of the present article is to generalize Theorem~\ref{temp_def} (obtained in the linear case), to the quasilinear case and to obtain ``as large as possible" nonnegative (optimal) weight $W$ satisfying
$$\mathcal{Q}(\varphi)\geq\gl  \int_{\Omega} W(x)|\varphi|^p\dnu \quad\forall \varphi\in\core.$$ In particular, we answer affirmatively a problem posed by the authors in \cite{DFP} (see Problem~13.12 therein).

The extension of Theorem~\ref{temp_def} to the quasilinear case is not a straightforward task. First, due to the nonlinearity of the operator $Q_V$, the supersolution construction has to be modified, and in particular in the case $p>n$, the supersolution construction leading to optimal potentials is essentially different. In fact, we could not extend Theorem~\ref{temp_def} to operators $Q_V$ with $V\neq 0$. Secondly, the proof of Theorem~\ref{temp_def} given in \cite{DFP} is mostly of linear nature, and therefore a new approach is needed for the quasilinear case. Moreover, the proof of Theorem~\ref{thm_opt_hardy} actually provides us with an alternative proof for parts (b) and (c) of Theorem~\ref{temp_def}. On the other hand, it seems that there is no analog to part (e) of Theorem~\ref{temp_def} concerning the essential spectrum of the corresponding operator. We note that in the linear case, the proof of part (e) relies on a construction of a family of generalized eigenfunctions, and this construction does not apply to the quasilinear case.

Let us introduce first our definition of {\em optimal Hardy-weights} for $Q_V$ in a punctured domain.
\begin{definition}\label{def_opt}{\em
Suppose that $\mathcal{Q}_V\geq 0$ in $\Gw$, and denote $\Omega^\star:=\Gw\setminus \{0\}$. Assume that a nonzero nonnegative function $W$ satisfies the following Hardy-type inequality \begin{equation}\label{opt_hardy4}
\mathcal{Q}_V(\varphi) \geq \gl \int_{\Omega^\star} W|\varphi|^p\dnu \qquad \forall \varphi\in C_0^\infty(\Omega^\star),
\end{equation}
where $\gl$ is a positive constant. Set $\lambda_0:=\lambda_0(Q_V,W,\Omega^\star)$.

We say that $W$ is an {\em optimal Hardy-weight} for the operator $Q_V$ in $\Gw$ if the following conditions hold true.
\begin{enumerate}

 \item The functional $\mathcal{Q}_{V-\lambda_0 W}$ is {\em critical in} $\Gw^\star$, i.e. for any $W_1\gneqq \gl_0 W$, the Hardy-type inequality
 $$ \mathcal{Q}_V(\varphi) \geq \int_{\Omega^\star} W_1 |\varphi|^p\dnu\qquad \forall \varphi\in C_0^\infty(\Omega^\star)$$
does not hold. In particular, the equation $Q_{V-\lambda_0 W}(u)=0$ in $\Omega^\star$ admits, up to a multiplicative positive constant, a unique positive (super)solution $v$; such a $v$ is called the {\em Agmon ground state}.

 \item  $\gl_0$ is also the best constant for inequality \eqref{opt_hardy4} restricted to functions $\vgf$ that are compactly supported either in a fixed punctured neighborhood of the origin, or in a fixed neighborhood of infinity in $\Gw$. In particular, $\lambda_\infty\left(Q_V,W,\Omega^\star \right)=\lambda_0.$

 \item Suppose further that $V\geq 0$. For an open set $\tilde{\Gw}\subset \Gw$, let $\mathcal{D}_{\mathcal{Q}_V}^{1,p}(\tilde{\Gw})$ be the completion of $C_0^\infty(\tilde{\Gw})$ with respect to the norm $\mathcal{Q}_V(\cdot)^{1/p}$. Then the functional $\mathcal{Q}_{V-\gl_0 W}$ is {\em null-critical} at $0$ and at infinity in the following sense: for any pre-compact open set $O$ containing $0$, the (Agmon) ground state $v$ of $Q_{V-\gl_0W}$ in $\Gw^\star$ belongs neither to $\mathcal{D}_{\mathcal{Q}_V}^{1,\,p}(\Omega\setminus \bar{O})$ nor to $\mathcal{D}_{\mathcal{Q}_V}^{1,\,p}(O\setminus\{0\})$.
  In particular, the variational problem
  \begin{equation}\label{var_prob}
\inf_{v\in \mathcal{D}^{1,\,p}(\Omega^\star)}\left\{\frac{\mathcal{Q}_V(\vgf)}{\int_{\Omega^\star}|\vgf|^pW\dnu}\right\}
\end{equation}
  does not admit a minimizer.
\end{enumerate}
 }
\end{definition}
\begin{remark}\label{ab_not_c}{\em
It is natural to ask whether all the above properties of an optimal Hardy-weight are independent. It is indeed the case; in fact, in \cite{DFP} we gave the following example which shows that, in general, (3) is not a consequence of (1) and(2).

Let $0\leq V\in C_0^\infty(\mathbb{R}^n)$ be a potential such that the operator $-\Gd-V(x)$ is critical in $\mathbb{R}^n$. Consider the operator $Q:=-\Gd+\mathbf{1}-V(x)$, and the potential $W(x):=\mathbf{1}$. Then $\gl_{0}(Q,W,\mathbb{R}^n)=\gl_{\infty}(Q,W,\mathbb{R}^n)=1$. On the other hand, the operator $Q-W$ is null-critical in $\mathbb{R}^n$ for $n\leq 4$, and positive-critical if $n>4$.
 }
\end{remark}

\begin{remark}\label{rem_conv}{\em

If $p\neq 2$, the definition of $\mathcal{D}_{\mathcal{Q}_V}^{1,p}(\Gw)$ cannot be applied to
the case where $V\not \geq 0$, since the positivity of the
functional $\mathcal{Q}_V$ on $\core$ does not necessarily imply its
convexity, and thus it does not give rise to a norm (see the discussion in \cite{ky6}).
 }
\end{remark}
Using a modified supersolution construction, we obtain the main result of our paper:
\begin{Thm}\label{thm_opt_hardy}
Let $\bar{\infty}$ denote the ideal point in the one-point compactification of $\Gw$. Suppose that $-\Delta_p$ admits a positive $p$-harmonic function $\mathcal{G}$ in $\Omega^\star:=\Gw\setminus \{0\}$ satisfying one of the following conditions \eqref{assumpt_7} and \eqref{assumpt_8}:

\begin{align}\label{assumpt_7}
 1<p\leq n, \;& \lim_{x\to 0}\mathcal{G}(x)=\infty, \qquad \mbox{and } \lim_{x\to\bar{\infty}}\mathcal{G}(x)=0,\\[2mm]
p>n,  \;& \lim_{x\to 0}\mathcal{G}(x)=\gg \geq 0, \;\; \mbox{and } \lim_{x\to\bar{\infty}}\mathcal{G}(x)=\left\{
                                       \begin{array}{ll}
                                         \infty & \hbox{if }\; \gg =0,\\[2mm]
                                         0 & \hbox{if } \gg\; > 0.
                                       \end{array}
                                     \right.
\label{assumpt_8}
\end{align}
Define a positive function $v$ and a nonnegative weight $W$ on $\Omega^\star$ as follows:
\begin{enumerate}
\item  If either \eqref{assumpt_7} is satisfied, or \eqref{assumpt_8} is satisfied with $\gg =0$, then $v:=\mathcal{G}^{(p-1)/p}$, and $W:=\left(\frac{p-1}{p}\right)^p\left|\frac{\nabla \mathcal{G}}{\mathcal{G}}\right|^p$.
\item If \eqref{assumpt_8} is satisfied with $\gg>0$, then $v:=[\mathcal{G}(\gamma-\mathcal{G})]^{(p-1)/p}$, and

$$W:=\left(\frac{p-1}{p}\right)^p\left|\frac{\nabla \mathcal{G}}{\mathcal{G}(\gg-\mathcal{G})}\right|^p|\gg-2\mathcal{G}|^{p-2}\left[2(p-2)\mathcal{G}(\gg-\mathcal{G})+\gg^2\right].$$

\end{enumerate}
\medskip
\noindent Then the following Hardy-type inequality holds in $\Omega^\star$:
\begin{equation}\label{opt_hardy}
 \int_{\Omega^{\star}}|\nabla \varphi|^p\dnu \geq \int_{\Omega^\star} W|\varphi|^p\dnu \qquad \forall \varphi\in C_0^\infty(\Omega^\star),
\end{equation}
and $W$ is an {\em optimal} Hardy-weight for $-\Gd_p$ in $\Gw$.

Moreover, up to a multiplicative constant, $v$ is the unique positive supersolution of the equation $Q_{-W}(w)=0$ in $\Gw^\star$.
\end{Thm}

\begin{remark}\label{assumptions}
{\em
Let us discuss hypotheses \eqref{assumpt_7} and \eqref{assumpt_8}. Suppose first that $\Gw$ is a $C^{1,\ga}$-bounded domain with $0<\ga\leq 1$. Let $G^\Gw(x,0)$ be the positive minimal $p$-Green function of the operator $-\Delta_p$ in $\Gw$ with a pole at $0$. Then $\mathcal{G}:=G^\Gw(\cdot,0)$ satisfies either \eqref{assumpt_7}, or  \eqref{assumpt_8} with $\gg>0$. This assertion follows, for example, from the results in \cite{FP,Garcia} and is valid more generally for any subcritical operator $Q_V$ with $V\in L^\infty(\Gw)$.

Suppose further that $\Omega$ is a $C^{1,\alpha}$-subdomain of a noncompact Riemannian manifold $M$ (where $\alpha\in(0,1]$), with a positive $p$-Green function $G^{M}$ that satisfies
$$\lim_{x\to \bar{\infty}} G^{M}(x,0)=0 .$$
Using a standard exhaustion argument, the monotonicity of the Green functions as a function of the domain, and the above remark, it follows that $\mathcal{G}:=G^\Gw(\cdot,0)$ satisfies either \eqref{assumpt_7}, or \eqref{assumpt_8} with $\gg>0$.

If $\Omega=\R^n$, $Q=-\Gd_p$, and $1<p<n$ (resp., $p>n$) , then $\mathcal{G}(x):=|x|^{\frac{p-n}{p-1}}$ satisfies assumption \eqref{assumpt_7} (resp., assumption \eqref{assumpt_8} with $\gg=0$). In this case, $\Omega^\star=\R^n\setminus \{0\}$ is the punctured space, and $W(x)=\left(\frac{p-1}{p}\right)^p\left|x\right|^{-p}$ is the classical Hardy potential. We note that the criticality of the operator
$$Q_{-W}(u)=-\div(|\nabla u|^{p-2}\nabla u)-\left(\frac{p-1}{p}\right)^p\frac{|u|^{p-2}u}{\left|x\right|^{p}}\qquad  \mbox{in }\Gw^\star$$
follows also from the proof of \cite[Theorem 1.3]{PS} given by Poliakovsky and Shafrir.
 }
\end{remark}

\begin{remark}\label{rem_ends}{\em
In our study, the domain $\Gw^\star$ should be viewed as a manifold with two ends: the origin and $\bar{\infty}$, the ideal point obtained by the one-point compactification of $\Gw$. In particular, the notion of optimal Hardy-weight can be extended analogously to the case of any manifold with two ends (see Section~\ref{sec_ann_ext}, for an extension of Theorem \ref{thm_opt_hardy} to annular or exterior domains).
}
\end{remark}
The outline of the present paper is as follows. In Section~\ref{sec_prelim} we review the theory of positive solutions for $p$-Laplacian type equations. Section~\ref{sec_coarea} is devoted to a coarea formula which is a key result in our study (see Proposition~\ref{pro_key}). Section~\ref{sec33} explains the supersolution construction of Hardy-weights in various situations. Section~\ref{sec_pf} is devoted to the proof of Theorem~\ref{thm_opt_hardy}. Section~\ref{sec_ann_ext} we present extensions of Theorem~\ref{thm_opt_hardy} to the case of annular and exterior domains.  In Section~\ref{sec_rellich} we present some $L^p$-Rellich-type inequalities and discuss the optimality of the obtained constants. Finally, in Section~\ref{sec34} we study the supersolution construction for general operators $Q_V$ of the form \eqref{eqQ}, where the obtained weight is in general not optimal.

\mysection{Preliminaries}\label{sec_prelim}

Let $\Omega$ be a domain in $\R^n$ (or in a noncompact Riemannian manifold of dimension $n$), where $n\geq 2$. We equip $\Gw$ with the one-point compactification, and denote by $\bar{\infty}$ the added ideal point which we call the \textit{infinity} in $\Omega$. So, $x_n\to \bar{\infty}$ if  and only if the sequence $\{x_n\}_{n\in\mathbb{N}}\subset \Gw$ eventually exits any compact subset of $\Omega$. For example, if $\Omega\subset \R^n$ is bounded, then the infinity in $\Gw$ is just $\partial\Omega$, and $x_n\to \bar{\infty}$ if and only if  $\mathrm{dist}(x_n,\partial \Omega)\to 0$, where $\mathrm{dist}(\cdot,\partial\Gw)$ is the distance function to $\partial \Gw$.

Throughout the paper we assume that $\Omega$ is equipped with an absolutely continuous measure $\nu$ with respect to the Lebesgue measure in $\R^n$ (or with respect to the Riemannian measure in the case of a Riemannian manifold), and that the corresponding density is positive and smooth.

We write $\Omega_1 \Subset \Omega_2$ if $\Omega_2$ is open, $\overline{\Omega_1}$ is
compact and $\overline{\Omega_1} \subset \Omega_2$. Let $f,g \in C(D)$ be nonnegative functions, we denote $f\asymp g$ on
$D$ if there exists a positive constant $C$ such that
$$C^{-1}g(x)\leq f(x) \leq Cg(x) \qquad \mbox{ for all } x\in D.$$

For $1<p<\infty$, we consider a quasilinear operator

\begin{equation}\label{P}
Q_V(u)=Q(u):=-\Gd_p(u)+V|u|^{p-2}u ,
\end{equation}
where $V\in L^\infty_{\mathrm{loc}}(\Omega)$. Here, the $p$-Laplacian $\Delta_p$ is defined by

$$\Delta_p(u):=\div(|\nabla u|^{p-2}\nabla u),$$
where $\div$ is the divergence  with respect to the measure $\nu$, so, the integration by parts formula
$$\int_\Omega -\div(X)\vgf\dnu = \int_\Omega  X\cdot\nabla \vgf \dnu $$
holds for any smooth vector field $X$ and function $\vgf$ that are compactly supported in $\Omega$. Associated to $Q_V$ there is the energy functional
\begin{equation}
\mathcal{Q}_V(\varphi)=\mathcal{Q}(\varphi):=\int_\Omega \big(|\nabla \varphi|^p+V|\varphi|^p\big)\dnu\qquad \varphi \in C_0^\infty(\Omega).
\end{equation}
We say that $u\in W^{1,\,p}_{\mathrm{loc}}(\Omega)$ is a (weak) {\em solution} of the equation $Q(u)=f$ in $\Gw$ if for every $\varphi \in C_0^\infty(\Omega)$,

\begin{equation}\label{weak_def}
\int_\Omega \left( |\nabla u|^{p-2}\nabla u\cdot \nabla \varphi+V|u|^{p-2}u\varphi\right)\dnu=\int_\Omega f\varphi\dnu\,.
\end{equation}
We define in a similar way the notions of {\em subsolution} and {\em supersolution} of $Q(u)=f$. Weak
solutions of the equation $Q(u)=0$ admit H\"older continuous first
derivatives, and nonnegative solutions of the equation  $Q(u)=0$
satisfy the Harnack inequality (see for example
\cite{LU,PuccS,Serrin1,Serrin2,T}). Therefore, in the definition \eqref{weak_def} with $f=0$, one can equivalently take test functions in $C^1_0(\Omega)$ instead of $C_0^\infty(\Omega)$.

The notions of criticality and subcriticality of $Q_V$ have been studied in this context, and we refer to \cite{PT1} for an account on this. For completeness, we recall the essential notions and results that we need throughout the present paper.

The operator $Q$ is said to be {\em nonnegative in $\Gw$} (and we denote it by $Q\geq 0$) if the equation $Q(u)=0$ in $\Gw$ admits a positive (super)solution.
As in the (selfadjoint) linear case, the following Allegretto-Piepenbrink type theorem holds:
\begin{Thm}[{\cite[Theorem~2.3]{PT1}}]\label{AAP}
$Q_V\geq 0$ in $\Gw$ if and only if $\mathcal{Q}_V(\varphi)\geq 0$ for every $\varphi\in C_0^\infty(\Omega)$.

\end{Thm}
Throughout the paper, we assume that $Q$ is nonnegative in $\Gw$. As in the linear case, there is a dichotomy for nonnegative operators:  $Q$ of the form \eqref{P} is either {\em critical} or {\em subcritical} in $\Gw$. We note that in the case of $Q=-\Delta_p$ on a Riemannian manifold $M$ equipped with its Riemannian measure, criticality (resp., subcriticality) is often called {\em $p$-parabolicity} (resp., {\em $p$-hyperbolicity}). Criticality/subcriticality has several equivalent definitions, which we recall below, but first we need to introduce some notions.

\begin{Def}
{\em

We say that a sequence $\{\varphi_k\}_{k\in\mathbb{N}}$ of nonnegative functions belonging to $C_0^\infty(\Omega)$ is a {\em null-sequence} for $\mathcal{Q}$ in $\Gw$ if there exists an open set $B\Subset \Omega$ such that

$$\lim_{k\to\infty} \mathcal{Q}(\varphi_k)=\lim_{k\to\infty} \int_\Omega \big(|\nabla \varphi_k|^p+V|\varphi_k|^p\big)\dnu=0,\quad \mbox{and } \int_B |\varphi_k|^p\dnu\asymp 1.$$
}
\end{Def}

\begin{Def}
{\em
Let $K_0$ be a compact set in $\Omega$.  A positive solution $u$
of the equation $Q(w)=0$ in $\Omega\setminus K_0$ is said to be a
{\it  positive solution of minimal growth in a neighborhood of
infinity in} $\Omega$ (or $u\in\M_{\Omega,K_0}$ for brevity) if
for any compact set $K$ in $\Omega$, with a smooth boundary, such
that $K_0 \Subset \mathrm{int}(K)$, and any positive supersolution
$v\in C((\Omega\setminus K)\cup
\partial K)$ of the equation $Q(w)=0$ in $\Omega\setminus K$,
the inequality $u\le v$ on $\partial K$ implies that $u\le v$ in
$\Omega\setminus K$.

Similarly, for $x_0\in \Gw$, we define the notion of a positive solution of the equation $Q(w)=0$ in a punctured neighborhood of $x_0$ of minimal growth at $x_0$.
}

\end{Def}

We have
\begin{Thm}[\cite{PT1,FP}]
Suppose that $\mathcal{Q}$ is nonnegative in $\Omega$, and fix $x_0\in \Omega$. Then the equation $Q(w)=0$ has (up to a multiplicative constant) a unique positive solution $u\in \M_{\Omega,\{x_0\}}$ of minimal growth in a neighborhood of
infinity in $\Omega$.

Moreover, $u$ is either a {\em global} positive  solution of $Q(w)=0$ in $\Gw$ (such a solution is called {\em Agmon's ground state}), or $u$ has singularity at $x_0$ with the following asymptotic:
$$
u(x) \underset{x\to x_0}{\sim}\; \left\{
    \begin{array}{ll}
        |x-x_0|^{\frac{p-n}{p-1}} & \mbox{if } 1<p< n, \\[2mm]
        -\log|x-x_0| & \mbox{if } p=n,\\[2mm]
        1  & \mbox{if } p>n.
    \end{array}
\right.
$$
In the latter case, the appropriately normalized solution is called the {\em positive minimal Green function of $Q$ in $\Gw$} with a pole at $x_0,$ and is denoted by $G_Q^\Gw(x,x_0)=G(x)$.

Furthermore, any positive solution $v$ of $Q(w)=0$ in a punctured neighborhood of $x_0$ of minimal growth at $x_0$ has the following asymptotic near $x_0$:
$$
v(x) \underset{x\to x_0}{\sim}\; \left\{
    \begin{array}{ll}
        1 & \mbox{if } 1<p\leq n, \\[2mm]
                |x-x_0|^{\frac{p-n}{p-1}}  & \mbox{if } p>n.
    \end{array}
\right.
$$
\end{Thm}
\begin{definition}\label{def_crit}{\em
Suppose that $Q\geq 0$ in $\Gw$. Then  $Q$ is said to be {\em critical} in $\Omega$ if the equation $Q(u)=0$ in $\Omega$ admits a (Agmon) ground state, and {\em subcritical} in $\Omega$ otherwise.
 }
\end{definition}
\begin{lemma}[\cite{PT1}]\label{lem_crit}
Suppose that $Q\geq 0$ in $\Gw$. Then the following assertions are equivalent:
\begin{enumerate}
\item $Q$ is {\em critical} in $\Omega$.
\item The equation $Q(w)=0$ in $\Omega$ admits a unique positive supersolution (up to a multiplicative constant).
\item  The only nonnegative function $W$ such that the inequality
$$\mathcal{Q}(\varphi)\geq \int_{\Omega} W(x)|\varphi|^p\dnu$$ holds for every $\varphi\in C_0^\infty(\Omega)$ is $W=0$.
\item  $\mathcal{Q}$ admits a null sequence in $\Gw$.
 \end{enumerate}
\end{lemma}
A nonnegative functional $\mathcal{Q}$ might contain an indefinite term (if the potential has indefinite sign). Although, by Picone identity \cite{AH}, such functional $\mathcal{Q}$ can be represented as the integral of a nonnegative Lagrangian $L$, this $L$ still contains an indefinite term. It was proved in \cite{PTT} that $\mathcal{Q}$ is equivalent to a {\em simplified energy} containing only nonnegative terms, as we explain now.
\begin{definition}\label{defse}{\em
Let $v$ be a positive solution of the equation $Q(u)=0$ in $\Gw$. The {\em simplified energy} is defined for nonnegative functions $w\in C_0^\infty(\Omega)$ by
\begin{equation}\label{eq_simp}
\!\mathcal{Q}_{\mathrm{sim}}(w)\!:=\!\!\left\{\!\!
  		\begin{array}{ll}
  			\displaystyle{\int_\Omega \!\!\Big(v^2|\nabla w|^2(v|\nabla w|+w|\nabla v|)^{p-2}\Big)\dnu} & \!\!\mbox{ if } 1<p\leq 2,\\[6mm]
			\displaystyle{\int_\Omega\! \Big(v^p|\nabla w|^p+v^2|\nabla v|^{p-2}w^{p-2}|\nabla w|^2\Big)\dnu} & \!\!\mbox{ if } p>2.
		\end{array}
\right.
    \end{equation}
 }
\end{definition}
Since Picone identity holds also on manifolds (cf. \cite[Section~2]{PTT}), it follows that Lemma 2.2 in \cite{PTT} is valid also on manifolds. Therefore, we obtain the following equivalence between the functional $\mathcal{Q}$ and the simplified energy $\mathcal{Q}_{\mathrm{sim}}$:
\begin{Lem}[{\cite[Lemma 2.2]{PTT}}]\label{simpl_energy}
Assume that $Q=Q_V\geq 0$ in $\Omega$. Let $v\in C_\mathrm{loc}^{1,\ga}(\Gw)$ be a fixed positive solution of the equation $Q(u)=0$ in $\Omega$. Then for all $w\in C_0^\infty(\Omega)$ we have
$$\mathcal{Q}(w) \asymp \mathcal{Q}_{\mathrm{sim}}\left(\frac{w}{v}\right).$$
\end{Lem}
Lemma \ref{simpl_energy} is a generalization of the \textit{ground state transform} (see \cite{DFP}) to the nonlinear case. In the nonlinear case, one obtains the equivalence (and not equality, as in the linear case) between $\mathcal{Q}$ and a functional containing only positive terms. As a corollary of Lemma~\ref{simpl_energy}, we state the following obvious upper estimate for the simplified energy, which will be of use later.

\begin{Lem}\label{simple_est}
Denote
$$X(w):=\int_{\Omega^\star} v^p |\nabla w|^p\dnu, \qquad Y(w):=\int_{\Omega^\star} |w|^p |\nabla v|^p\dnu.$$
Then there exists $C>0$ such that for all $w\in \core$ we have
\begin{equation}\label{est_qsim}
\mathcal{Q}_{\mathrm{sim}}(w)\leq\left\{
  		\begin{array}{ll}
  			CX(w) & \mbox{ if } \;1<p\leq 2,\\\\
			C\left[X(w)+\left(\frac{X(w)}{Y(w)}\right)^{2/p}Y(w)\right] & \mbox{ if }\; p>2.
		\end{array}
\right.
\end{equation}

\end{Lem}
We conclude this section with the following useful lemma
\begin{Lem}\label{weak_lapl}
Let $u\in C^{1,\alpha}_{\mathrm{loc}}(\Omega)$ for some $\alpha\in(0,1)$, and $f\in C^2$. Then the following formula holds in the weak sense:
\begin{equation}\label{weak_eq}
-\Delta_p(f(u))=-|f'(u)|^{p-2}\left[(p-1)f''(u)|\nabla u|^p+f'(u)\Delta_p(u)\right].
\end{equation}
\end{Lem}
\begin{proof}
Denote $g:=-\Delta_p(u)$, and let $\varphi\in C_0^\infty(\Omega)$. Then, by Leibniz's product rule and the chain rule we have
\begin{multline*}
\int_\Omega |\nabla f(u)|^{p-2}\nabla f(u)\cdot \nabla \varphi\dnu =\\[2mm]
\int_\Omega\!\! |\nabla u|^{p\!-\!2}\nabla u\cdot \nabla \!\left(|f'(u)|^{p\!-\!2}\!f'(u)\varphi\right)\!\!\dnu
-\!\!\int_\Omega \!\!|\nabla u|^{p\!-\!2}\nabla u\cdot \nabla \!\left(|f'(u)|^{p\!-\!2}\!f'(u)\right)\!\!\varphi\!\dnu
\end{multline*}
Note that for $p\geq 2$, the function $\psi(s):=|s|^{p-2}s$ is continuously differentiable, and $\psi'(s):=(p-1)|s|^{p-2}$. Moreover, for $1<p<2$ the function $\psi$ is not differentiable at zero but its derivative near zero is integrable. Recall that by our assumptions $u\in C^{1,\alpha}(\Omega)$. Therefore if $p\geq 2$, then the function $|f'(u)|^{p-2}f'(u)\varphi$ belongs to $C^{1}_0(\Omega)$. On the other hand, for $1<p<2$, $\nabla \big(|f'(u)|^{p-2}f'(u)\varphi\big)$ is integrable. Hence in both cases, $|f'(u)|^{p-2}f'(u)\varphi$ is a legitimate test function. Consequently,
$$\int_\Omega |\nabla u|^{p-2}\nabla u\cdot \nabla \left(|f'(u)|^{p-2}f'(u)\varphi\right)\dnu=\int_\Omega g|f'(u)|^{p-2}f'(u)\varphi\dnu.$$
Therefore,
\begin{multline*}
\int_\Omega |\nabla f(u)|^{p-2}\nabla f(u)\cdot \nabla \varphi\dnu =\\[2mm]
\int_\Omega g|f'(u)|^{p-2}f'(u)\varphi\dnu -\int_\Omega |\nabla u|^{p-2}\nabla u\cdot \nabla \left(|f'(u)|^{p-2}f'(u)\right)\varphi \dnu.
\end{multline*}
Consequently, in the weak sense we have
$$-\Delta_p(f(u))=-|\nabla u|^{p-2}\nabla u\cdot \nabla \left(|f'(u)|^{p-2}f'(u)\right)-\Delta_p(u)|f'(u)|^{p-2}f'(u).$$
But since $\psi'(s):=(p-1)|s|^{p-2}$ for $s\neq 0$, and $\psi'$ is integrable at $0$, we have that in the weak sense
\begin{equation}\label{derive_2}
|\nabla u|^{p-2}\nabla u\cdot \nabla \left(|f'(u)|^{p-2}f'(u)\right)=(p-1)|f'(u)|^{p-2}|\nabla u|^pf''(u).
\end{equation}
This completes the proof of Lemma~\ref{weak_lapl}.
\end{proof}

\mysection{The coarea formula}\label{sec_coarea}
The present section is devoted to the proof of a coarea formula associated with the $p$-Laplacian. It seems that this key result in our study cannot be extended to the case of an operator $Q_V$ of the form \eqref{eqQ} with $V\neq 0$ and $p\neq 2$ (cf. \cite[Lemma~9.2]{DFP}, where an analogue coarea formula is obtained for any linear symmetric operator).
\begin{Pro}\label{pro_key}
Let $\mathcal{G}$ be a positive $p$-harmonic function in $\Gw^\star:=\Omega\setminus\{0\}$.
Define $v:=\mathcal{G}^{(p-1)/p}$. Then there exists positive constants $c$ and $\tilde{c}$ such that for every real functions $f$ and $g$, defined on $(0,\infty)$ such that $f(v)$ and $g(v)$ have compact support in $\Omega$, the following formulae hold:
\begin{equation}\label{IPP}
\int_{\Omega^\star} f(v) |\nabla v|^p \dnu=c\int_{\inf v}^{\sup v} \frac{f(\gt)}{\gt}\,\mathrm{d}\gt,
\end{equation}
and
\begin{equation}\label{IPP2}
\int_{\Omega^\star} g(\mathcal{G}) |\nabla \mathcal{G}|^p \dnu=\tilde{c}\int_{\inf \mathcal{G}}^{\sup \mathcal{G}} g(t)\dt.
\end{equation}
\end{Pro}
\begin{proof}
The idea is the same as in  \cite[Lemma 9.2]{DFP}. Setting $g(t):=f(t^{(p-1)/p})$ and performing the change of variable $\gt:=t^{(p-1)/p}$, it follows that \eqref{IPP} is equivalent to \eqref{IPP2}. By the coarea formula, we have
\begin{multline}\label{gW}
	\int_{\Omega^\star}g(\mathcal{G})|\nabla \mathcal{G}|^p\dnu =\int_{\inf  \mathcal{G}}^{\sup\mathcal{G}} \left(\int_{\{\mathcal{G}=t\}}g(\mathcal{G})\frac{|\nabla \mathcal{G}|^p}{|\nabla \mathcal{G}|}\mathrm{d}\sigma\right)\mathrm{d}t\\[2mm]
		=\int_{\inf \mathcal{G}}^{\sup\mathcal{G}} g(t)\left(\int_{\{\mathcal{G}=t\}}|\nabla \mathcal{G}|^{p-1}\mathrm{d}\sigma\right)\mathrm{d}t,
		\end{multline}
where $\mathrm{d}\sigma$ denotes the Hausdorff measure of dimension $n-1$.
Indeed, $\mathcal{G}\in C^{1,\alpha}_{\mathrm{loc}}$, in particular $|\nabla \mathcal{G}|^{p-1}\in L^1_{\mathrm{loc}}$ and the use of the coarea formula is licit.  We claim that $\int_{\{\mathcal{G}=t\}}|\nabla \mathcal{G}|^{p-1}\mathrm{d}\sigma$ does not depend on $t$. This essentially follows from Green's formula, but since $\mathcal{G}$ is not smooth, we have to be careful. Let us fix $t_1, t_2$ such that $\inf\mathcal{G}<t_1<t_2<\sup\mathcal{G}$, and define $\mathcal{A}$ to be the ``annulus''
$$\mathcal{A}:=\{x\in\Gw^\star \mid t_1<\mathcal{G}<t_2\}.$$
The boundary of $\mathcal{A}$ is the disjoint union of $\partial_-:=\{\mathcal{G}=t_1\}$ and of $\partial_+:=\{\mathcal{G}=t_2\}$. We claim that $\mathcal{A}$ has \textit{finite perimeter}, i.e., $\chi_{\mathcal{A}}$, the characteristic function of $\mathcal{A}$, has bounded variation. Indeed,
$$\chi_{\mathcal{A}}=\chi_{(t_1,t_2)}\circ \mathcal{G},$$
therefore,
$$\nabla \chi_{\mathcal{A}}=\left(\chi_{(t_1,t_2)}'(\mathcal{G})\right)\,\nabla \mathcal{G}=(\delta_{\mathcal{G}=t_1}-\delta_{\mathcal{G}=t_2})\nabla \mathcal{G}.$$
Since $\nabla \mathcal{G}$ is continuous, we obtain that $\chi_{\mathcal{A}}\in BV$, hence $\mathcal{A}$ has finite perimeter. Since $|\nabla \mathcal{G}|^{p-2}\nabla \mathcal{G}$ is continuous, and has divergence which vanishes in $\mathcal{A}$ in the weak sense, Theorems 5.2 and 7.2 in \cite{Chen} imply that the Gauss-Green formula is valid on $\mathcal{A}$:
\begin{equation}\label{Gauss}
0\!=\!-\int_{\mathcal{A}}\div(|\nabla \mathcal{G}|^{p-2}\nabla \mathcal{G})\dnu\!=
\!\int_{\partial_+^\star}|\nabla \mathcal{G}|^{p-2}\nabla \mathcal{G}\cdot\mathbf{n}\,\mathrm{d}\sigma+\int_{\partial_-^\star}|\nabla \mathcal{G}|^{p-2}\nabla \mathcal{G}\cdot\mathbf{n}\,\mathrm{d}\sigma,
\end{equation}
where $\partial_+^\star$ and $\partial_-^\star$ are the \textit{reduced boundaries} (see \cite{Chen}), $\mathbf{n}$ is the \textit{measure theoretic exterior unit normal}, and $\sigma$ is the $(n-1)$-dimensional Hausdorff measure. If $x\in \partial_+$ (resp., $x\in \partial_-$) is such that $|\nabla \mathcal{G}(x)|\neq 0$, then the boundary of $\mathcal{A}$ is $C^1$ in a neighborhood of $x$, and the vector field ${\nabla \mathcal{G}}/{|\nabla \mathcal{G}|}$ is well-defined near $x$; it is equal to $\mathbf{n}$ (resp., $-\mathbf{n}$) around $x$. Furthermore, we can write around $x$

\begin{equation}\label{normal}
|\nabla \mathcal{G}|^{p-1}=|\nabla \mathcal{G}|^{p-2}|\nabla \mathcal{G}|=\pm|\nabla \mathcal{G}|^{p-2}\nabla \mathcal{G}\cdot \mathbf{n}\quad \mbox{ if }x\in \partial_{\pm}.
\end{equation}
Since $\mathcal{G}$ is $C^{1,\alpha}$, we may use a generalization of Sard's theorem due to Bojarski, Haj{\l}asz and Strzelecki \cite{BHS} to infer that for
almost every $t\in (0,\infty)$
$$\sigma\left(\{\mathcal{G}=t\}\cap \mathrm{Crit}(\mathcal{G})\right)=0,$$
where $\mathrm{Crit}(\mathcal{G})$ is the set of critical points of $\mathcal{G}$. This implies  that for almost all $t$ ,  \eqref{normal} holds $\sigma$-almost everywhere on $\{\mathcal{G}=t\}$, and that for almost all $t_1$ and  $t_2$,  the reduced boundaries  $\partial_+^\star$ and $\partial_-^\star$ coincide with $\pd_+=\{\mathcal{G}=t_2\}$ and $\pd_-=\{\mathcal{G}=t_1\}$, respectively,  up to a set of zero measure for $\sigma$.  Since $|\nabla \mathcal{G}|^{p-1}$ and $|\nabla \mathcal{G}|^{p-2}\nabla \mathcal{G}$ are continuous, we obtain that for almost all $t_1$ and $t_2$ we have
\begin{multline*}
\int_{\partial_+^\star}|\nabla \mathcal{G}|^{p-2}\nabla \mathcal{G}\cdot\mathbf{n}\,\mathrm{d}\sigma+\int_{\partial_-^\star}|\nabla \mathcal{G}|^{p-2}\nabla \mathcal{G}\cdot\mathbf{n}\,\mathrm{d}\sigma=\\
\int_{\pd_+}|\nabla \mathcal{G}|^{p-1}\mathrm{d}\sigma-\int_{\pd_-}|\nabla \mathcal{G}|^{p-1}\mathrm{d}\sigma,
 \end{multline*}
and therefore by \eqref{Gauss},
$$\int_{\{\mathcal{G}=t_2\}}|\nabla \mathcal{G}|^{p-1}\,\mathrm{d}\sigma=\int_{\{\mathcal{G}=t_1\}}|\nabla \mathcal{G}|^{p-1}\,\mathrm{d}\sigma.$$
Thus, $\int_{\{\mathcal{G}=t\}}|\nabla \mathcal{G}|^{p-1}\,\mathrm{d}\sigma$ is equal (almost everywhere) to a constant independent of $t$.

\end{proof}

\mysection{The supersolution construction for the $p$-Laplacian}\label{sec33}
In this section, we show how to extend the \textit{supersolution construction}, which was a primary tool in the study of the linear case in \cite{DFP}, to the $p\,$-Laplace operator. As in the linear case, in some cases this construction will give us \textit{optimal} Hardy weights. We postpone the study of the supersolution construction for $Q_V$ with $V\neq 0$ to Section~\ref{sec34}, and here we present two particular supersolution constructions which apply to the $p$-Laplace operator. These constructions  will lead us to the optimal weights of Theorems~\ref{thm_opt_hardy}.

For completeness, we recall the supersolution construction for linear (not necessarily symmetric) elliptic operators:
\begin{lemma}[{\cite[Theorem~3.1]{crit2}} and {\cite[Remark~5.4]{DFP}}]\label{lem_conv}
Let $P$ be a second-order {\em linear} elliptic operator with real coefficients defined in $\Gw$. For  $j=0,1$, let $V_j$ be real valued potentials, and suppose that  $v_j$ are positive (super)solutions of the equations $(P +V_j)u = 0$ in $\Gw$. Then for $0\leq \ga \leq 1$ the function
$$v_\alpha:=(v_1))^{\ga}(v_0)^{1-\ga}$$ is a positive (super)solution of the linear equation
\begin{equation}\label{Pga3}
\left[P +(1-\alpha)V_0 + \alpha V_1 - \alpha(1-\alpha)W\right]u =0 \qquad \mbox{in } \Gw,
\end{equation}
where
\begin{equation}\label{eq_W}
W:=\left|\nabla \log\left(\frac{v_0}{v_1}\right)\right|_A^2,
\end{equation}
 $A=A(x)$ is the nonnegative definite matrix associated with the principal part of the operator $P$, and for $\xi\in\R^n$, $|\xi|_A^2 := \xi \cdot A \xi$.
\end{lemma}
We notice that since the proof of Lemma~\ref{lem_conv} is purely local and algebraic, we obtain in fact the following {\em pointwise} result.
\begin{Cor}\label{cor_conv}
Let $P$ be a second-order {\em linear} elliptic operator with real coefficients defined in $\Gw$. For  $j=0,1$, let $V_j$ be real valued potentials, and suppose that  $v_j$ are positive functions satisfying the differential  (in)equality $$(P +V_j)v_j \; \displaystyle{\substack{= \\[1mm] (\geq)}}\; 0\qquad \mbox{at } x_0\in \Gw.$$  Then for $0\leq \ga \leq 1$ the function
$v_\alpha=(v_1)^{\ga}(v_0)^{1-\ga}$ satisfies the differential  (in)equality
\begin{equation}\label{Pga31}
\left[P +(1-\alpha)V_0 + \alpha V_1 - \alpha(1-\alpha)W\right]u \; \displaystyle{\substack{= \\[1mm] (\geq)}}\; 0  \qquad \mbox{at } x_0\in \Gw,
\end{equation}
where $W$ is the function defined by \eqref{eq_W}.
\end{Cor}
A related -- but weaker -- convexity result is known in the case of $p$-Laplacian type equations:
\begin{lemma}[{\cite[Proposition~4.3]{PT1}}]\label{Prop2_ky3}
 Let $V_0, V_1\in L^\infty_{\mathrm{loc}}(\Omega)$,
$V_0\neq V_1$. For $\ga\in [0,1]$ we denote \be
Q_\ga(u):=Q_{(1-\ga)V_0 +\ga V_1}(u)=(1-\ga)Q_{V_0}(u)+\ga Q_{V_1}(u),\ee and suppose that
$Q_{V_i}\geq 0$ in $\Gw$ for $i=0,1$.

Then  $Q_\ga\geq 0$ in $\Gw$ for all $\ga\in[0,1]$.
Moreover, $Q_\ga$ is subcritical in
$\Omega$ for all $\ga\in(0,1)$.
\end{lemma}

\begin{remark}\label{rem_9}{\em
Lemma~\ref{Prop2_ky3} does not provide us with an explicit nonzero Hardy-weight $W_\ga$ for $Q_\ga$, although the subcriticality of $Q_\ga$ ensures the existence of a strictly positive weight.
 }
\end{remark}
The supersolution construction has been extended to the $p$-Laplacian itself by several authors, with $v_\alpha:=(v_1)^{\ga}(v_0)^{1-\ga}$,  in the particular case where $v_0$ is a positive $p$-harmonic function, and $v_1=\mathbf{1}$ (see for example \cite{AS,DAD,DFP} and references therein). In particular, the following Caccioppoli-type inequality has been obtained in \cite{DFP}:
\begin{Pro}[{\cite[Proposition 13.11]{DFP}}]\label{supersolution p-Laplacian}
Assume that $\mathcal{G}$ is a positive supersolution (resp., solution) of the equation $-\Delta_p(w)=0$ in $\Gw$. Then for $\alpha\in (0,1)$, $\mathcal{G}^\alpha$ is a positive  supersolution (resp., solution) of the equation $Q_{-W_\alpha}(w)=0$ in $\Gw$, where
$$W_\alpha:=\alpha^{p-1}(1-\alpha)(p-1)\left|\frac{\nabla \mathcal{G}}{\mathcal{G}}\right|^p.$$
In particular,  by taking the optimal value $\alpha=\frac{p-1}{p}$ we obtain the following logarithmic Caccioppoli inequality:
\begin{equation}\label{cacc}
\int_\Omega |\nabla\varphi|^p\dnu\geq \left(\frac{p-1}{p}\right)^p\int_\Omega \left|\frac{\nabla v}{v}\right|^p |\varphi|^p\dnu \qquad \forall\varphi\in C_0^\infty(\Omega),
\end{equation}
where $v$ is any positive $p$-superharmonic function in $\Gw$.
\end{Pro}
\begin{proof}
The first assertion of the proposition follows from Lemma~\ref{weak_lapl} and in particular from \eqref{weak_eq} with $f(s):=s^\ga$.
Hence using the Allegretto-Piepenbrink theorem \ref{AAP}, we obtain \eqref{cacc}.
\end{proof}
\begin{remark}\label{rem_Cacc}{\em
Inequality \eqref{cacc} has been independently proved in \cite{DAD} by L.~D'Ambrosio and S.~Dipierro, using a different approach.
 }
\end{remark}
\begin{example}\label{ex4} {\em
Consider Proposition~\ref{supersolution p-Laplacian} in the particular case $\Omega=\mathbb{R}^n\setminus \{0\}$, $p\neq n$, and $\mathcal{G}(x)=|x|^{\frac{p-n}{p-1}}$. Then \eqref{cacc}  clearly implies the classical Hardy inequality (with the best constant):

\begin{equation}\label{hardy_cla}
 \int_{\R^n\setminus\{0\}} |\nabla \vgf|^p\dx  \geq \left|\frac{p-n}{p}\right|^p\int_{\R^n\setminus \{0\}} \frac{|\vgf(x)|^p}{|x|^p}\,\mathrm{d}x \qquad \forall \vgf\in \core.
\end{equation}
 }
\end{example}
We will see later that Proposition \ref{supersolution p-Laplacian} yields an optimal Hardy weight if $\mathcal{G}$ further satisfies either assumption \eqref{assumpt_7} or \eqref{assumpt_8} with $\gg=0$ (see Theorem \ref{thm_opt_hardy}). However, as we shall see in Section \ref{sec34}, this supersolution construction does not provide us with an optimal Hardy weight if $\Omega$ is a bounded, $C^{1,\alpha}$-domain if $\mathcal{G}$ satisfies \eqref{assumpt_8} with $\gg>0$. In this case and also in other cases (see Section~\ref{sec_ann_ext}), an \textit{optimal} Hardy weight will be obtained using a different supersolution construction given by the following proposition.
\begin{Pro}\label{super_p>n}
Suppose that $\mathcal{G}$ is a $C^{1,\gb}$-positive supersolution (resp., solution) of $-\Delta_pw=0$ in $\Omega$ satisfying $0\leq m<\mathcal{G}<M<\infty$ in $\Omega$, where $0<\gb\leq 1$.

\noindent Set $v_\ga:=[(\mathcal{G}-m)(M-\mathcal{G})]^\alpha$, and define
\begin{multline}\label{wgg1gg2}
W_\alpha:=
(p\!-\!1)\alpha^{p\!-\!1}\!\!\left|\frac{\nabla \mathcal{G}}{v_1}\right|^p\!\!|m\!+\!M\!-\!2\mathcal{G}|^{p\!-\!2}\!\!\left[2(2\alpha\!-\!1)v_1\!+\!
(1\!-\!\alpha)(M\!-\!m)^2\right]\!\geq \!0.
\end{multline}
Then for $\ga$ satisfying
$$\ga\in \left\{
             \begin{array}{ll}
               \,[1/2,1] & \mbox{ if } m>0, \\[2mm]
               \,[0,1] & \mbox{ if } m=0,
             \end{array}
           \right.
 $$
the function $v_\ga$ is a positive supersolution (resp., solution) of the equation $Q_{-W_\alpha}(w)=0$ in $\Gw$.

\noindent In particular, let $\alpha=(p-1)/p$, and assume that either $\alpha=(p-1)/p\geq 1/2$, or $m=0$. Define
\be\label{wgg1gg21}
W:=W_{\frac{p-1}{p}}=\left(\frac{p-1}{p}\right)^p\left|\frac{\nabla \mathcal{G}}{v_1}\right|^p|m+M-2\mathcal{G}|^{p-2}\left[2(p-2)v_1+(M-m)^2\right].
\ee
Then $$v:=v_{\frac{p-1}{p}}=[(\mathcal{G}-m)(M-\mathcal{G})]^{\frac{p-1}{p}}$$ is a positive solution (resp., supersolution) of $Q_{-W}(w)=0$ in $\Omega$, and the following $L^p$-Hardy type inequality holds:
\begin{equation}\label{Hardy_p>n}
\int_{\Omega}|\nabla \vgf|^p\dnu \geq \int_{\Omega}W|\vgf|^p\dnu \qquad \forall \vgf\in C_0^\infty(\Omega).
\end{equation}
\end{Pro}

\begin{proof}
Let $0\leq \ga \leq 1$. By our assumption, $\mathcal{G}\in C^{1,\gb}_{\mathrm{loc}}(\Omega)$ for some $\gb\in(0,1]$. Moreover, the function $f(s)=\left[(s-m)(M-s)\right]^{\alpha}$ belongs to $C^2\big((0,\gg)\big)$. Consequently, one may apply Lemma \ref{weak_lapl} with $\mathcal{G}$ and $f$ to obtain that in the weak sense,
$$ -\Delta_p(v_\ga)\; \displaystyle{\substack{= \\[1mm] (\geq)}}\; -(p-1)|f'(\mathcal{G})|^{p-2}|\nabla \mathcal{G}|^pf''(\mathcal{G}) = W_\ga v_\ga^{p-1}\qquad \mbox{in } \Omega. $$
Therefore, $v_\ga\!=\!f(\mathcal{G})$ is a positive (super)solution of the equation $Q_{-W_\ga}(w)\!=\!0$ in $\Omega$, and the Allegretto-Piepenbrink type theorem (Theorem \ref{AAP})
implies
\begin{equation*}
\int_{\Omega}|\nabla \vgf|^p\dnu \geq \int_{\Omega}W_\ga|\vgf|^p\dnu \qquad \forall \vgf\in C_0^\infty(\Omega).
\end{equation*}
In particular, for $\alpha=(p-1)/p$ we have \eqref{Hardy_p>n}.
\end{proof}
\begin{Rem}
{\em
Let $\Gw_1\Subset\Gw_2\subset \R^n$ be two open sets. Suppose that $\Gw:=\Gw_2\setminus\Gw_1$ is a $C^{1,\gb}$-bounded annular-type domain such that $\pd \Gw$ is the union of  $\Gg_1=\pd\Gw_1$, and $\Gg_2=\pd\Gw_2$.  Let $\mathcal{G}$ be the solution of the Dirichlet problem
$$\left\{
    \begin{array}{ll}
      -\Gd_p(u)=0 & \hbox{in } \Gw,\\[2mm]
      u=m & \hbox{on } \Gg_1,\\[2mm]
 u=M & \hbox{on } \Gg_2,
    \end{array}
  \right.
$$
where $0\leq m < M$. Then $\mathcal{G}$ satisfies the assumptions of Proposition~\ref{super_p>n}.

Moreover, if $p>n$, $\Gw$ is a $C^{1,\gb}$-bounded domain with $0<\gb\leq 1$, and $\mathcal{G}:=G^\Gw(\cdot,0)$ is the positive minimal $p$-Green function of the operator $-\Delta_p$ in $\Gw$ with a pole at $0$. Then $\mathcal{G}$ satisfies the assumptions of Proposition~\ref{super_p>n} in $\Gw^\star$, with $m:=\lim_{x\to \pd \Gw}\mathcal{G}(x)=0$, and $M:=\lim_{x\to 0}\mathcal{G}(x)$.
}
\end{Rem}

\begin{Rem}
{\em If in Proposition~\ref{super_p>n} the supersolution $\mathcal{G}$ is unbounded and satisfies $\mathcal{G}>m$ in $\Omega$, then one should simply consider the supersolution construction with $v_\ga:= (\mathcal{G}-m)^\ga$ with $0\leq \ga\leq 1$ to obtain the Hardy-type inequality
\begin{equation}\label{Hardy_p>n12}
\int_{\Omega}|\nabla \vgf|^p\dnu \geq \left(\frac{{p-1}}{p}\right)^p\int_{\Omega} \left|\frac{\nabla \mathcal{G}}{\mathcal{G}-m}\right|^p |\vgf|^p\dnu \qquad \forall \vgf\in C_0^\infty(\Omega)
\end{equation}
(cf. Proposition~\ref{supersolution p-Laplacian}).
}
\end{Rem}

\begin{Rem}
{\em
A new phenomenon appears in Proposition~\ref{super_p>n}: if $p\neq 2$,  then the weight $W_\ga$ {\em necessarily vanishes} in $\Omega$. Indeed, $W_\ga=0$ on the set $$\left\{x\in \Gw\mid \mathcal{G}(x)=\frac{m+M}{2}\right\}.$$
}

\end{Rem}


\mysection{Proof of theorems~\ref{thm_opt_hardy}}\label{sec_pf}
The present section is devoted to the proof of the main result of the paper, namely Theorem~\ref{thm_opt_hardy}, that deals with the case $V=0$, and  claims the optimality of the supersolution construction for the $p$-Laplacian in $\Gw^\star$. We divide the proofs into three parts: the criticality of $\mathcal{Q}_{-W}$, the optimality of the constant near infinity and zero, and finally the null-criticality of $\mathcal{Q}_{-W}$.

\subsection{Criticality}
In the present subsection, we prove the criticality of $\mathcal{Q}_{-W}$. We divide the proof into two parts, according to which of the assumptions \eqref{assumpt_7}, \eqref{assumpt_8} is satisfied. We start by showing the criticality of $\mathcal{Q}_{-W}$ if either \eqref{assumpt_7} or \eqref{assumpt_8} with $\gg=0$ is  satisfied. This is a consequence of the following proposition:
\begin{Pro}\label{pro_crit}
Assume that in Theorem~\ref{thm_opt_hardy} the positive $p$-harmonic function $\mathcal{G}$ satisfies
\begin{equation}\label{assumpt_71}
\left\{
  \begin{array}{ll}
\displaystyle{\lim_{x\to 0} \mathcal{G}=\infty \;\mbox{ and }\; \lim_{x\to \bar{\infty}}} \mathcal{G}=0 & \hbox{ if } 1<p\leq n,\\[6mm]
\displaystyle{\lim_{x\to 0} \mathcal{G}=0 \;\;\;\mbox{ and }\; \lim_{x\to \bar{\infty}}} \mathcal{G}=\infty & \hbox{ if } p>n.
  \end{array}
\right.
\end{equation}
Then the functional $\mathcal{Q}_{-W}$ is critical in $\Omega^\star$.
\end{Pro}

\begin{proof}

Let $v:=\mathcal{G}^{\frac{p-1}{p}}$. Proposition~\ref{supersolution p-Laplacian} implies that $v$ is a positive solution of the equation $-\Gd_p(w)-W|w|^{p-2}w=0$ in $\Omega^\star$. We construct a null-sequence for the functional $\mathcal{Q}_{-W}$ in a similar fashion as in the proof of \cite[Theorem~1.3]{PS}.
Let
$$\varphi_n(t):=\left\{
  		\begin{array}{lll}
  			 0                         &\;\; 0\leq t\leq \frac{1}{n^2},\\[2mm]
			  2+\dfrac{\log t}{\log n} &\;\; \frac{1}{n^2}\leq t\leq \frac{1}{n},\\[2mm]
			  1                        & \;\;  \frac{1}{n}\leq t\leq n,\\[2mm]
			  2-\dfrac{\log t}{\log n} &\;\; n\leq t\leq n^2,\\[2mm]
			  0                        &\;\; t\geq n^2.
		\end{array}
\right.$$
Set $w_n:=\varphi_n(v)$, and consider the sequence $\{vw_n\}_{n\in\mathbb{N}}$.
\medskip

\noindent{\bf Claim:}  $\{vw_n\}$ is a null-sequence for the functional $\mathcal{Q}_{-W}$.

\medskip

Set $B:=\{x\in \Gw^\star\mid 1< v< 2\}$, then $\bar{B}$ is compact in $\Omega^\star$.
By Lemma \ref{simpl_energy} we have
$$\mathcal{Q}_{-W}(vw)\asymp \mathcal{Q}_{\mathrm{sim}}(w),$$
where $\mathcal{Q}_{\mathrm{sim}}$ is the simplified energy for the functional $\mathcal{Q}_{-W}$ (see \eqref{eq_simp}).
Thus, we need to prove that
\begin{equation}\label{eq_null}
\lim_{n\to\infty}\frac{\mathcal{Q}_{\mathrm{sim}}(w_n)}{\int_{B}(vw_n)^p\dnu}=0.
\end{equation}
Set $$X_n:=X(w_n)=\int_{\Omega^\star}v^p|\nabla w_n|^p\dnu,\quad \mbox{and } Y_n:=Y(w_n)=\int_{\Omega^\star}w_n^p|\nabla v|^p\dnu.$$
Using the coarea formula \eqref{IPP}, we obtain
\begin{multline*}
X_n=c_1\int_{\Omega^\star}v^{p}|\varphi_n'(v)|^p |\nabla v|^p\dnu=
c\int_0^\infty (t|\varphi_n'(t)|)^p\frac{\mathrm{d}t}{t}=\\[5mm]
 c\left(\frac{1}{\log n}\right)^p\left(\int_{\frac{1}{n^2}}^{\frac{1}{n}}\frac{\mathrm{d}t}{t}+\int_n^{n^2}\frac{\mathrm{d}t}{t}\right)
 =2c\left(\frac{1}{\log n}\right)^{p-1}.
\end{multline*}
Using again \eqref{IPP}, we get

\begin{equation*}
Y_n=\int_{\Omega^\star}w_n^p|\nabla v|^p\dnu =c\int_0^\infty |\varphi_n(t)|^p\frac{\mathrm{d}t}{t}
 \asymp \int_{\frac{1}{n}}^{n}\frac{\mathrm{d}t}{t}
 \asymp \log n.
\end{equation*}
On the other hand, we clearly have
$$\int_B (vw_n)^p\dnu\asymp 1.$$
Recall that by \eqref{est_qsim}, the simplified energy can be estimated from above by
$$\mathcal{Q}_{\mathrm{sim}}(w_n)\leq C\left\{
  		\begin{array}{ll}
  			X_n & \mbox{ if } 1<p\leq 2,\\\\
			X_n+\left(\frac{X_n}{Y_n}\right)^{2/p}Y_n & \mbox{ if } p>2.
		\end{array}
\right.$$
Therefore, $\lim_{n\to\infty}\mathcal{Q}_{\mathrm{sim}}(w_n) =0$, and  \eqref{eq_null} is proved.  Thus, $\{vw_n\,:\,n\in \mathbb{N}\}$ is a null-sequence for the functional $\mathcal{Q}_{-W}$, and $\mathcal{Q}_{-W}$ is critical in $\Gw^\star$.
\end{proof}
Next, we prove the criticality of $\mathcal{Q}_{-W}$ if assumption \eqref{assumpt_8} with $\gg>0$ is satisfied:
\begin{Pro}\label{pro_crit2}
Assume that in Theorem~\ref{thm_opt_hardy} $p>n$, and the positive $p$-harmonic function $\mathcal{G}$ satisfies
\begin{equation}\label{assumpt_81}
  \lim_{x\to 0} \mathcal{G}=\gg>0 \;\mbox{ and }\; \lim_{x\to \bar{\infty}} \mathcal{G}=0.
  \end{equation}
 Then the functional $\mathcal{Q}_{-W}$ is critical in $\Omega^\star$.

\end{Pro}

\begin{proof}

The proof follows closely the proof of Proposition \ref{pro_crit}. Assume for simplicity that $\gg=\mathcal{G}(0)=1$. Recall that $v:=\left[\mathcal{G}(1-\mathcal{G})\right]^{\frac{p-1}{p}}$. Proposition~\ref{super_p>n} implies that $v$ is a positive solution of the equation $-\Gd_p(w)-W|w|^{p-2}w=0$ in $\Omega^\star$. We construct a null-sequence for the functional $\mathcal{Q}_{-W}$. This time, let
$$\varphi_n(t):=\left\{
  		\begin{array}{lll}
  			 0                         &\;\; 0\leq t\leq \frac{1}{n^2},\\[2mm]
			  2+\dfrac{\log t}{\log n} &\;\; \frac{1}{n^2}\leq t\leq \frac{1}{n},\\[2mm]
			  1                        & \;\;  \frac{1}{n}\leq t,\\[2mm]
		\end{array}
\right.$$
and consider the sequence $\{w_n=\varphi_n(v)\}_{n\in\mathbb{N}}$. By hypothesis, $v(0)=0$ and $\lim_{x\to\bar{\infty}}v(x)=0.$ Therefore, for every $n\in \mathbb{N}$, $w_n$ is compactly supported in $\Omega^\star$.
\medskip

\noindent{\bf Claim:}  The sequence $\{vw_n\}_{n\in\mathbb{N}}$ is a null-sequence for the functional $\mathcal{Q}_{-W}$.

\medskip

Set $B:=\{x\in \Gw^\star\mid \frac{1}{4}< v< \frac{3}{4}\}$, then $\bar{B}$ is compact in $\Omega^\star$.
As in the proof of Proposition \ref{pro_crit}, we set

$$X_n:=X(w_n)=\int_{\Omega^\star}v^p|\nabla w_n|^p\dnu,\quad \mbox{and } Y_n:=Y(w_n)=\int_{\Omega^\star}w_n^p|\nabla v|^p\dnu.$$
Let $f(s):=\left[s(1-s)\right]^{\frac{p-1}{p}}$. Using the coarea formula \eqref{IPP2}, we obtain
\begin{align*}\label{eq89}
X_n&=\int_{\Omega^\star}v^p|\nabla v|^p|\varphi_n'(v)|^p\dnu\\[2mm]
&=C\int_{\Omega^\star}\left[\mathcal{G}(1-\mathcal{G})\right]^{p-2}|1-2\mathcal{G}|^p|\varphi_n'\circ f(\mathcal{G})|^p|\nabla \mathcal{G}|^p\dnu\\[2mm]
&=\frac{C}{(\log n)^p}\int_{f(t)\in [1/n^2,1/n]}\frac{|1-2t|^p}{t(1-t)}\dt \asymp \frac{1}{(\log n)^{p-1}}\,.
\end{align*}
Using again the coarea formula \eqref{IPP2}, we get
\begin{align*}
Y_n=\int_{\Omega^\star}(\varphi_n(v))^p|\nabla v|^p\dnu=\int_0^1\varphi_n(f(t))^p\frac{|1-2t|^p}{t(1-t)}\dt \asymp \log n .
\end{align*}
In light of \eqref{est_qsim} we have $\lim_{n\to\infty}\mathcal{Q}_{\mathrm{sim}}(w_n) =0$. On the other hand, we clearly have
$$\int_B (vw_n)^p\dnu\asymp 1.$$
Hence, $\{vw_n\}_{n\in\mathbb{N}}$ is a null-sequence for the functional $\mathcal{Q}_{-W}$.
\end{proof}

\subsection{Optimality of the constant near infinity and zero}
In the present subsection we prove the optimality of the constant $C_p:=\left(\frac{p-1}{p}\right)^p$ near the ends of $\Gw^\star$. As in the previous subsection, we split the proof into two parts.
\begin{Pro}\label{best_const1}
Assume that in Theorem~\ref{thm_opt_hardy} the positive $p$-harmonic function $\mathcal{G}$ satisfies
\begin{equation}\label{assumpt_72}
\left\{
  \begin{array}{ll}
\displaystyle{\lim_{x\to 0} \mathcal{G}=\infty \;\mbox{ and }\; \lim_{x\to \bar{\infty}}} \mathcal{G}=0 & \hbox{ if } 1<p\leq n,\\[6mm]
\displaystyle{\lim_{x\to 0} \mathcal{G}=0 \;\;\;\mbox{ and }\; \lim_{x\to \bar{\infty}}} \mathcal{G}=\infty & \hbox{ if } p>n.
  \end{array}
\right.
\end{equation}
 Then the constant $\gl= C_p$ in the Hardy inequality
\begin{equation}\label{hardy_opt}
\int_{\Omega^\star}|\nabla \varphi|^p\dnu \geq    \gl \int_{\Omega^\star} \left|\frac{\nabla \mathcal{G}}{\mathcal{G}}\right|^p |\varphi|^p\dnu
\end{equation}
is also the best constant for functions $\vgf$ compactly supported either in a fixed punctured neighborhood of the origin, or  in a fixed neighborhood of infinity in $\Gw$.
\end{Pro}
\begin{proof}
We assume that $1<p\leq n$, and present the proof of the optimality at infinity, the other cases are proved similarly. We proceed by a contradiction.

Suppose that there exists a positive constant $\lambda$ and a compact set $K\Subset \Omega$ containing zero such that
\begin{equation}\label{const_non_opt}
\int_{\Omega\setminus K}\Big(|\nabla \psi|^p-W|\psi|^p\Big)\dnu \geq \lambda \int_{\Omega\setminus K} W|\psi|^p \dnu \qquad \forall \psi\in C_0^\infty(\Omega\setminus K).
\end{equation}
We apply inequality \eqref{const_non_opt} to $\psi=v\varphi$, where $v:=\mathcal{G}^{(p-1)/p}$ is a positive solution of $Q_{-W}(w)=0$, and $\varphi\in C_0^\infty(\Omega\setminus K)$. Now, use Lemma~\ref{simpl_energy} and \eqref{const_non_opt} to obtain that for some positive constant $\beta$ we have
\begin{equation}\label{const_non_opt2}
\beta Y(\varphi)\leq\left\{
  \begin{array}{ll}
    X(\varphi) & \mbox{ if } 1<p\leq 2,\\[4mm]
	X(\varphi)+\left(\frac{X(\varphi)}{Y(\varphi)}\right)^{\frac{2}{p}}Y(\varphi) & \mbox{ if } p>2,
  \end{array}
 \right. \quad \forall \varphi\in C_0^\infty(\Omega\setminus K),
\end{equation}
where we recall that $X(\varphi):=\int_{\Omega^\star} v^p |\nabla \varphi|^p\dnu$ and $Y(\varphi):=\int_{\Omega^\star}\varphi^p|\nabla v|^p\dnu=\int_{\Omega^\star}v^p\varphi^p W\dnu$. In the case $p>2$, using the fact that for every $\varepsilon>0$, there is a constant $C>0$ such that for every $t>0$, $t+t^{2/p}\leq Ct+\varepsilon$, we have that

$$X(\varphi)+\left(\frac{X(\varphi)}{Y(\varphi)}\right)^{\frac{2}{p}}Y(\varphi)\leq CX(\varphi)+\varepsilon Y(\varphi).$$
Taking $\varepsilon<\beta$, we get by \eqref{const_non_opt2} that for any $1<p<\infty$, there is a constant $C>0$ such that
\begin{equation}\label{opt_p>2}
CY(\varphi)\leq X(\varphi) \qquad \forall \varphi\in C_0^\infty(\Omega\setminus K).
\end{equation}
Assume without loss of generality that $\{v\leq 1\}\subset \Omega\setminus K$. Using the coarea formula \eqref{IPP}, and applying inequality \eqref{opt_p>2} to $\varphi=\phi(v)$, where $\phi\in C_0^\infty\big((0,1)\big)$ we get that
\begin{equation}\label{ineq_1D}
\int_0^1 |\phi(t)|^p\frac{\mathrm{d}t}{t}\leq C\int_0^1(t|\phi'(t)|)^p\frac{\mathrm{d}t}{t}\qquad\forall \phi\in C_0^\infty\big((0,1)\big).
\end{equation}
But by \cite[Theorem~1 of Sec.~1.3.2]{Maz}, this inequality cannot hold.

Alternatively, an easy way to see that \eqref{ineq_1D} does not hold is to define a sequence $\{\phi_\varepsilon\}$ of compactly supported Lipschitz continuous functions in $(0,1)$ of the form
$$\phi_\varepsilon(t):=\left\{
\begin{array}{ll}
 \dfrac{t}{\varepsilon |\log \varepsilon|^\gg} & t\in(0,\varepsilon), \\[5mm]
 \dfrac{1}{|\log t|^\gamma} & t\in (\varepsilon,\frac{1}{2}),\\[5mm]
\psi(t) & t\in (\frac{1}{2},1),
\end{array}
\right.
$$
where $\psi$ is a smooth function, independent of $\varepsilon$ such that $\psi(1)=0$, and $\gamma>0$ will be determined later. Apply inequality \eqref{ineq_1D} to $\phi_\varepsilon$ to get
\begin{equation}\label{limit_cont}
\int_\varepsilon^{\frac{1}{2}} |\phi_\varepsilon (t)|^p\frac{\mathrm{d}t}{t}\leq C\left(\frac{1}{\varepsilon |\log\varepsilon|^\gamma}\int_0^\varepsilon t^p\frac{\mathrm{d}t}{t}+\int_\varepsilon^1(t|\phi_\varepsilon'(t)|)^p\frac{\mathrm{d}t}{t}\right).
\end{equation}
Since $p>1$,
$$\lim_{\varepsilon\to0}\frac{1}{\varepsilon |\log\varepsilon|^\gamma}\int_0^\varepsilon t^p\frac{\mathrm{d}t}{t}=\lim_{\varepsilon\to0}\frac{\varepsilon^{p-1}}{|\log\varepsilon|^\gamma}=0,$$
therefore, letting $\varepsilon\to0$ in \eqref{limit_cont}, we get
$$\int_0^{\frac{1}{2}} \left(\frac{1}{|\log t|^\gamma}\right)^p\frac{\mathrm{d}t}{t}\leq C\left(\int_0^\frac{1}{2}\left(\frac{1}{|\log t|^{\gamma+1}}\right)^p\frac{\mathrm{d}t}{t}+\int_\frac{1}{2}^1(t\psi'(t))^p \frac{\mathrm{d}t}{t}\right).$$
The right-hand side is finite for every positive value of $\gamma$, since $p(\gamma+1)>1$. The left-hand side, on the contrary, is finite if and only if  $p\gamma>1$. Thus, taking $\gamma$ such that $p\gamma\leq 1$, we get a contradiction. As a consequence, inequality \eqref{ineq_1D} cannot hold.
\end{proof}
Next, we prove the optimality of the constant $C_p=\left(\frac{p-1}{p}\right)^p$ near the ends of $\Gw^\star$ if assumption \eqref{assumpt_8} with $\gg>0$ is satisfied:
\begin{Pro}\label{best_const2}
Assume that in Theorem~\ref{thm_opt_hardy} $p>n$, and the positive $p$-harmonic function $\mathcal{G}$ satisfies
\begin{equation}\label{assumpt_82}
  \lim_{x\to 0} \mathcal{G}=\gg>0 \;\mbox{ and }\; \lim_{x\to \bar{\infty}} \mathcal{G}=0.
  \end{equation}
 Denote
 $$V:=\left|\frac{\nabla \mathcal{G}}{\mathcal{G}(\gg-\mathcal{G})}\right|^p|\gg-2\mathcal{G}|^{p-2}\left[2(p-2)\mathcal{G}(\gg-\mathcal{G})+\gg^2\right].$$
Then in the Hardy inequality
\begin{equation}\label{hardy_opt2}
\int_{\Omega^\star}|\nabla \varphi|^p\dnu \geq    \gl\int_{\Omega^\star} V |\varphi|^p\dnu
\end{equation}
the constant $\gl= C_p$ is also the best constant for functions compactly supported either in a fixed punctured neighborhood of the origin, or  in a fixed neighborhood of infinity in $\Gw$.
\end{Pro}

\begin{proof}
We prove the optimality of the constant $C_p$ at infinity, the proof of the optimality at zero is similar (by replacing $\mathcal{G}$ with $(\gg-\mathcal{G})$). Note that
$W=C_pV$. Assume by contradiction that $C_p$ is not optimal at infinity, then there is a positive constant $\lambda$ and a compact subset $K$ of $\Omega$ containing $0$, such that

\begin{equation}\label{const_non_opti2}
\int_{\Omega\setminus K}\Big(|\nabla \psi|^p-W|\psi|^p\Big)\dnu \geq \lambda \int_{\Omega\setminus K} W|\psi|^p \dnu \qquad \forall \psi\in C_0^\infty(\Omega\setminus K).
\end{equation}
Since by our assumption $\lim_{x\to\bar{\infty}}\mathcal{G}(x)=0$, we have

$$W\underset{x\to \bar{\infty}}{\sim}\; \left(\frac{p-1}{p}\right)^{p}\left|\frac{\nabla \mathcal{G}}{\mathcal{G}}\right|^p.$$
Therefore, by enlarging $K$, we may assume that the following inequality is satisfied, for some $\mu>0$:

\begin{equation}\label{non-opti2}
\int_{\Omega\setminus K}\Big(|\nabla \psi|^p-W|\psi|^p\Big)\dnu \geq \mu \int_{\Omega\setminus K} \left|\frac{\nabla \mathcal{G}}{\mathcal{G}}\right|^p|\psi|^p \dnu \qquad \forall \psi\in C_0^\infty(\Omega\setminus K).
\end{equation}
We apply this inequality to $\psi=\varphi v$, where $v=[\mathcal{G}(\gg-\mathcal{G})]^{\frac{p-1}{p}}$ is a positive solution of $Q_{-W}(w)=0$, and $\varphi\in C_0^\infty(\Omega\setminus K)$. Define $\tilde{v}:=\mathcal{G}^{\frac{p-1}{p}}$, and notice that at infinity,

$$v\underset{x\to \bar{\infty}}{\sim}\; \tilde{v},$$
and

$$\left|\frac{\nabla v}{v}\right|^p\underset{x\to \bar{\infty}}{\sim}\; \left(\frac{p-1}{p}\right)^p\left|\frac{\nabla \mathcal{G}}{\mathcal{G}}\right|^p
=\left|\frac{\nabla \tilde{v}}{\tilde{v}}\right|^p.$$
Therefore, from Lemma \ref{simpl_energy}, \eqref{est_qsim} with $p>2$, and \eqref{non-opti2}, one gets that for some positive constant $\beta$,

\begin{equation}
\beta \tilde{Y}(\varphi)\leq \tilde{X}(\varphi)+\left(\frac{\tilde{X}(\varphi)}{\tilde{Y}(\varphi)}\right)^{\frac{2}{p}}\tilde{Y}(\varphi),\qquad \forall \varphi\in C_0^\infty(\Omega\setminus K),
\end{equation}
where $\tilde{X}(\varphi):=\int_{\Omega^\star} \tilde{v}^p |\nabla \varphi|^p\dnu$ and $\tilde{Y}(\varphi):=\int_{\Omega^\star}\varphi^p|\nabla \tilde{v}|^p\dnu$. We are back to inequality \eqref{const_non_opt2} of the proof of Proposition \ref{best_const1}, where we have shown that such an inequality cannot hold. Consequently, \eqref{const_non_opti2} does not hold, and the constant $\left(\frac{p-1}{p}\right)^p$ in the inequality \eqref{hardy_opt2} is optimal at infinity.

\end{proof}

\subsection{Null-criticality}
The null-criticality of the operators $Q_{-W}$ in $\Gw^\star$ follows from our coarea formula \eqref{IPP}. First, we have:
\begin{Pro}\label{null1}
Assume that in Theorem~\ref{thm_opt_hardy} the positive $p$-harmonic function $\mathcal{G}$ satisfies
\begin{equation}\label{assumpt_73}
\left\{
  \begin{array}{ll}
\displaystyle{\lim_{x\to 0} \mathcal{G}=\infty \;\mbox{ and }\; \lim_{x\to \bar{\infty}}} \mathcal{G}=0 & \hbox{ if } 1<p\leq n,\\[6mm]
\displaystyle{\lim_{x\to 0} \mathcal{G}=0 \;\;\;\mbox{ and }\; \lim_{x\to \bar{\infty}}} \mathcal{G}=\infty & \hbox{ if } p>n.
  \end{array}
\right.
\end{equation}
Then the functional $\mathcal{Q}_{-W}$ is null-critical at $0$ and at infinity in $\Gw$.
\end{Pro}
\begin{proof}
Let $v:=\mathcal{G}^{\frac{p-1}{p}}$, and denote also $u:=\mathcal{G}$. A minimizer of the variational problem \eqref{var_prob} is necessarily a positive solution of the equation $Q_{-W}=0$ in $\Gw^*$. Since $Q_{-W}$ is critical, a minimizer in $\mathcal{D}^{1,p}(\Gw)$ should be the ground state $v$. We claim that for any neighborhood $O$ of $0$, the ground state $v$ belongs neither to  $\mathcal{D}^{1,p}(\Omega\setminus \bar{O})$ nor to $\mathcal{D}^{1,p}(O\setminus\{0\})$. Indeed,
the coarea formula \eqref{IPP} implies that
$$\int_{\{t_-<u(x)<t_+\}}\hspace{-.7cm}|\nabla v|^p\dnu=c_1\int_{t_-}^{t_+} \frac{\dt}{t}
\underset{t_\pm \to \varepsilon_\pm}{\longrightarrow} \infty  ,$$
with $\varepsilon_+=\infty$ and $\varepsilon_-=0$. Thus, the claim is proved.
\end{proof}
The corresponding result, under assumption \eqref{assumpt_8} with $\gg>0$, reads as follows

\begin{Pro}\label{null2}
Assume that in Theorem~\ref{thm_opt_hardy} $p>n$, and the positive $p$-harmonic function $\mathcal{G}$ satisfies
\begin{equation}\label{assumpt_83}
  \lim_{x\to 0} \mathcal{G}=\gg>0 \;\mbox{ and }\; \lim_{x\to \bar{\infty}} \mathcal{G}=0.
  \end{equation}
 Then the functional $\mathcal{Q}_{-W}$ is null-critical at $0$ and at infinity in $\Gw$.
\end{Pro}

\begin{proof}
The proof is similar to the proof of Proposition~\ref{null1}. Indeed, recall that $v:=[\mathcal{G}(\gg-\mathcal{G})]^{\frac{p-1}{p}}$. Let $\varepsilon_+=\gg$ and $\varepsilon_-=0$. It is enough to prove that
$$\lim_{t_{\pm}\to\varepsilon_\pm}\int_{\{t_-<\mathcal{G}<t_+\}}\hspace{-.7cm}|\nabla v|^p\dnu=\infty.$$
We prove it when $t_-\to 0$, the other case is similar, (replace $\mathcal{G}$ with $(\gg-\mathcal{G})$). Define $\tilde{v}=\mathcal{G}^{\frac{p-1}{p}}$. At infinity in $\Omega$, we have

$$v\underset{x\to \bar{\infty}}{\sim}\;\tilde{v},$$
and

$$\left|\frac{\nabla v}{v}\right|^p\underset{x\to \bar{\infty}}{\sim}\; \left(\frac{p-1}{p}\right)^p\left|\frac{\nabla \mathcal{G}}{\mathcal{G}}\right|^p
=\left|\frac{\nabla \tilde{v}}{\tilde{v}}\right|^p.$$
Therefore,  by the coarea formula \eqref{IPP}, one has as $t_-\to 0$,

$$\int_{\{t_-<\mathcal{G}<\gg/2\}}\hspace{-.7cm}|\nabla v|^p\dnu\sim \int_{\{t_-<\mathcal{G}<\gg/2\}}\hspace{-.7cm}|\nabla \tilde{v}|^p\dnu=\int_{t_-}^{\frac{\gg}{2}}\frac{\mathrm{d}t}{t},$$
and consequently,

$$\lim_{t_{-}\to\varepsilon_-}\int_{\{t_-<\mathcal{G}<\gg/2\}}\hspace{-.7cm}|\nabla v|^p\dnu=\infty.$$

\end{proof}
We conclude the present section with a corollary concerning Caccioppoli inequality. Recall the logarithmic Caccioppoli inequality \eqref{cacc} which holds in particular in $\Gw^\star$:
\begin{equation}\label{cacc1}
 \int_{\Omega^\star} |\nabla\varphi|^p\dnu \geq \gm \int_{\Omega^\star} \left|\frac{\nabla v}{v}\right|^p |\varphi|^p\dnu \qquad \forall\varphi\in C_0^\infty(\Omega^\star),
\end{equation}
where $v$ is any positive $p$-superharmonic functions  in $\Omega^\star$, and $\gm\geq C_P=\left(\frac{p-1}{p}\right)^p$.
By the results of \cite{DFP} it follows that in the linear case (where $p=2$) the constant $C_2=1/4$ in \eqref{cacc1} is optimal.

Now, Theorem~\ref{thm_opt_hardy} clearly implies the optimality of the constant $C_p$ also for any $1<p\leq n$. More precisely, we have.
\begin{corollary}
Assume that $1<p\leq n$, and suppose that $\Gw$ is a $C^{1,\ga}$-domain of a noncompact Riemannian manifold $M$ (where $\alpha\in(0,1]$), and $-\Delta_p$ is subcritical in $M$. Let  $G^{M}$ be the positive minimal Green function, and assume that $\lim_{x\to \bar{\infty}} G^{M}(x,0)=0$.

Then the best constant in the logarithmic Caccioppoli inequality \eqref{cacc1} equals to  $\left(\frac{p-1}{p}\right)^p$.
\end{corollary}
\mysection{Optimal weights for annular and exterior domains}\label{sec_ann_ext}
In the present section we extend our main result (Theorem~\ref{thm_opt_hardy}), obtained for punctured domains, to two additional types of domains: annular-type domains and exterior-type domains.
As in the case of punctured domains, we view these two types of domains as manifolds with two ends. In particular, Definition~\ref{def_opt} of optimal Hardy-type weight (which was given for a punctured domain) is extended naturally to handle annular-type and exterior-type domains.

We assume that the given positive $p$-harmonic function admits limits at the two ends (one limit might be infinity). We use the supersolution constructions obtained in propositions~\ref{supersolution p-Laplacian} and \ref{super_p>n}, and the techniques used in the proof of Theorem~\ref{thm_opt_hardy} to obtain optimal Hardy-weights for these cases. We omit the proofs since they differ only slightly from the proof of Theorem~\ref{thm_opt_hardy}.

\begin{Thm}\label{thm_opt_hardym}
Let $\Gw$ be a $C^{1,\alpha}$ domain for some $\alpha>0$. Let $U\Subset \Omega$ be an open $C^{1,\alpha}$ subdomain of $\Omega$, and consider $\tilde{\Gw}:=\Gw\setminus U$. Denote by $\bar{\infty}$ the infinity in $\Omega$, and assume that $-\Delta_p$ admits a positive $p$-harmonic function $\mathcal{G}$ in $\tilde{\Omega}$ satisfying the following conditions
\begin{equation}\label{assumpt_7m}
\lim_{x\to \partial U}\mathcal{G}(x)=\gg_1, \; \lim_{x\to\bar{\infty}}\mathcal{G}(x)= \gg_2, \end{equation}
where $\gg_1\neq \gg_2$, and  $0\leq \gg_1,\, \gg_2 \leq \infty$. Denote $$m:=\min\{\gg_1,\gg_2\}, \qquad M:=\max\{\gg_1,\gg_2\}.$$

\noindent Define positive functions $v_1$ and  $v$, and a nonnegative weight $W$ on $\tilde{\Omega}$ as follows:

\medskip

(a) If $M<\infty$, assume further that either $m=0$ or $p\geq 2$, and let
$$v_1:=(\mathcal{G}-m)(M-\mathcal{G}), \qquad  v:=v_1^{(p-1)/p}=[(\mathcal{G}-m)(M-\mathcal{G})]^{(p-1)/p},$$
and
\be\label{wgg1gg21m}
W:=\left(\frac{p-1}{p}\right)^p\left|\frac{\nabla \mathcal{G}}{v_1}\right|^p|m+M-2\mathcal{G}|^{p-2}\left[2(p-2)v_1+(M-m)^2\right],
\ee

(b) If $M=\infty$, define
$$ v_1:=(\mathcal{G}-m), \qquad  v:=v_1^{(p-1)/p}=(\mathcal{G}-m)^{(p-1)/p},$$
and
\be\label{wgg1gg21m2}
W:=\left(\frac{p-1}{p}\right)^p\left|\frac{\nabla \mathcal{G}}{v_1}\right|^p.
\ee

\medskip

\noindent Then the following Hardy-type inequality holds true
\begin{equation}\label{opt_hardym}
 \int_{\tilde{\Omega}}|\nabla \varphi|^p\dnu \geq \int_{\tilde{\Omega}} W|\varphi|^p\dnu \qquad \forall \varphi\in C_0^\infty(\tilde{\Omega}),
\end{equation}
and $W$ is an {\em optimal} Hardy-weight for $-\Gd_p$ in $\tilde{\Gw}$.

Moreover, up to a multiplicative constant, $v$ is the unique positive supersolution of the equation $Q_{-W}(w)=0$ in $\tilde{\Gw}$.
\end{Thm}

\mysection{Optimal $L^p$ Rellich-type inequalities}\label{sec_rellich}
Throughout the present section we consider a \textit{linear} operator $P$. In \cite{DFP} we proved the following $L^2$-Rellich-type inequality.
\begin{lemma}[{\cite[Corollary~10.3]{DFP}}]\label{cor1_hardy_rell}
Assume that $P$ is a subcritical linear Schr\"odinger-type operator in $\Gw$ of the form
$$P:=-\div(A(x)\nabla\cdot)+V(x),$$ and let $v_0$ and $v_1$ be two linearly independent positive solutions of the equation $Pu=0$ in $\Gw$. Let $W:=\frac{1}{4}\left|\nabla \log\left(\frac{v_0}{v_1}\right)\right|_A^2$ be the Hardy-weight obtained by the supersolution construction with a pair $(v_0,v_1)$ (see \eqref{eq_W}). Suppose that $W$ is strictly positive, and fix $0\leq \gl\leq 1$. Then
\begin{itemize}
   \item[(a)] For a fixed  $0\leq \ga<  1$ and all $\vgf\!\in\! C_0^\infty(\Gw)$ the following Rellich-type inequality holds true
    \begin{equation}\label{exp8}
\int_{\Gw} \frac{|P\vgf|^2}{W(x)}\left(\frac{v_0}{v_1}\right)^\ga \!\dnu \geq \gl\left(1- \ga^2\right)^2  \int_{\Gw}  |\vgf|^2W(x)\left(\frac{v_0}{v_1}\right)^\ga \! \dnu
   .
\end{equation}
\item[(b)] If $P-W$ is critical in $\Gw$,  then  $\gl=1$ is the best constant in \eqref{exp8}.
\end{itemize}
\end{lemma}
We are interested in generalizing Lemma~\ref{cor1_hardy_rell}, and prove $L^p$-Rellich-type inequalities for the operator $P$. Our result hinges on the following $L^p$-Rellich-type inequality of E.B.~Davies and A.M.~Hinz:
\begin{theorem}[{\cite[Theorem 4]{DH}}]\label{thm_DH}
Let $\Gw$ be a domain in a Riemannian manifold of dimension $n \geq 2$, and let $1\leq p<\infty$. If $0 < v\in C(\Gw)$  with $-\Gd v > 0$ and $-\Gd (v^\gd) \geq 0$ for some $\gd > 1$,
then
$$ \int_\Gw \frac{v^p}{|\Gd v|^{p-1}} |\Gd \vgf|^p \dnu  \geq \frac{[(p-1)\gd +1]^p}{p^{2p}} \int_\Gw |\Gd v| |\vgf|^p \dnu\qquad \forall \vgf\in \core.
$$
\end{theorem}
If $P=-\div(A(x)\nabla\cdot)$ (i.e., $V=0$), Theorem~\ref{thm_DH} implies the following $L^p$-Rellich-type inequality:
\begin{Thm}\label{thm_Lp_rell}
Let $P:=-\div(A\nabla\cdot)$ be a subcritical operator in $\Gw$, and let $v_0$ be a positive (super)solution of the equation $Pu=0$ in $\Gw$ and $v_1:=\mathbf{1}$.  Let $W:=\frac{1}{4}\left|\nabla \log v_0\right|_A^2$ be the Hardy-weight obtained by the supersolution construction with a pair $(v_0,v_1)$, and suppose that $W>0$.  Then for every $\alpha\in (0,1)$ and $1\leq p<\infty$ the following Rellich-type inequality holds:
\begin{equation}\label{eq_rell_lp}
    \int_\Omega  \frac{|P\vgf|^p}{W^{p-1}}(v_0)^\alpha\!\dnu \!\geq \!\frac{4^p(1-\alpha)^p(p-1+\alpha)^p}{p^{2p}}\!\int_\Omega\! |\vgf|^pW(v_0)^\alpha\!\dnu\quad \forall \vgf\in C_0^\infty(\Omega).
\end{equation}
\end{Thm}
\begin{proof}
Apply Theorem~\ref{thm_DH}, with $v:=(v_0)^\ga$, and $\gd=1/\ga$. Since $-\Gd v\geq 4\ga(1-\ga)W v>0$, and $-\Gd v^\gd\geq 0$, we obtain \eqref{eq_rell_lp}.
\end{proof}
Using the ground state transform with a positive solution $v_1$, Theorem~\ref{thm_Lp_rell} implies:
\begin{Thm}\label{thm_Lp_rell1}
Let $P:=-\div(A\nabla\cdot)+V$ be a subcritical linear Schr\"odinger-type operator in $\Gw$, and let $v_0$ and, $v_1$ be two positive solutions of the equation  $Pu=0$ in $\Gw$.
Let $W:=\frac{1}{4}\left|\nabla \log \left(v_0/v_1\right)\right|_A^2$ be the Hardy-weight obtained by the supersolution construction with a pair $(v_0,v_1)$, and suppose that $W>0$.  Then for every $\alpha\in (0,1)$ and $1\leq p<\infty$ the following $L^p$-Rellich-type inequality holds:
\begin{equation}\label{eq_rell_lp1}
    \int_\Omega \!\! \frac{|P\vgf|^p}{W^{p-1}}\!\left(\frac{v_0}{v_1}\right)^\alpha\! v_1^{2-p}\dnu
    \geq \!\frac{4^p(1-\alpha)^p(p-1+\alpha)^p}{p^{2p}}\!\int_\Omega\! |\vgf|^pW\left(\frac{v_0}{v_1}\right)^\alpha\!\! v_1^{2-p} \dnu\quad
\end{equation}
for all $\vgf\in C_0^\infty(\Omega)$.
\end{Thm}
\begin{Rem}\label{rem_71}{\em
In the case $p=2$, we recover the best constant $\left(1- \ga^2\right)^2$ obtained in Lemma~\ref{cor1_hardy_rell}. We note that for $p\neq 2$, the constant of the $L^p$-Rellich-type inequalities \eqref{eq_rell_lp} and \eqref{eq_rell_lp1} is optimal at least in the classical case, where $\Gw=\R^n\setminus\{0\}$, $P=-\Gd$, $v_0=|x|^{2-n}$ and $v_1=\mathbf{1}$. The optimality of the constant in this case follows from the remark in \cite[page 521]{DH}.
}
\end{Rem}

\mysection{The supersolution construction for $Q_V$}\label{sec34}
In the present section we study the supersolution construction for operators $Q_V$ of the form \eqref{eqQ} under the assumption that (roughly speaking) the supersolutions $v_j$ have the same level sets. In Appendix~\ref{appendix1} we present a proof of the particular case of radially symmetric potentials.

The following result generalizes Lemma~\ref{lem_conv} for $p \neq 2$.
\begin{theorem}\label{supersol_cons}
Let $v_j$,  $j=0,1$, be two positive, linearly independent, $C^2$-(super)solutions of the equation $Q_{V_j}(u)=0$ in $\Omega$. Assume that $\nabla v_0$ does not vanish in $\Omega$, and that  $v_1=\varphi_1(v_0)$ for some $C^2$-function $\vgf_1$ such that $\varphi_1'(u)\neq 0$. For  $0\leq \ga\leq 1$, define the function $$v_\ga:= v_1^{\ga}v_0^{1-\ga},$$
and let
 $$V_\ga:=\Big((1-\ga)V_0|\nabla\log v_0|^{2-p}+\ga V_1|\nabla \log v_1|^{2-p} \Big)|\nabla\log v_\ga|^{p-2},$$
 $$W_\ga:= \ga(1-\ga)(p-1)
 \left|\nabla \log\left(\frac{v_0}{v_1}\right)\right|^2|\nabla\log v_\ga|^{p-2},$$
Then $v_\ga$ is a positive (super)solution of the equation
\begin{equation}\label{eqQv-Wa2_o1}
  Q_{V_\ga-W_\ga}(u)= 0 \qquad \mbox{ in } \Gw,
\end{equation}
and the following improved inequality holds
$$\mathcal{Q}_{V_\ga}(\vgf) \geq \int_\Gw W_\ga |\vgf|^{p}\dnu \qquad \forall \vgf\in C_0^\infty(\Omega). $$
\end{theorem}

\begin{remark}\label{rem12}{\em
If both $v_0$ and $v_1$ do not admit critical points, then the condition $v_1=\varphi_1(v_0)$ is equivalent to the fact that $\nabla v_0$ and $\nabla v_1$ are collinear at every point, and also to the fact that the level sets of $v_0$ and $v_1$ coincide, that is, for every $t_0>0$, there is $t_1>0$ such that
$$\{x\in\Gw\mid v_0(x)=t_0\}=\{x\in\Gw\mid v_1(x)=t_1\},$$
and vice versa. A particular case appears when $v_j$ are radially symmetric positive supersolutions (see Appendix~\ref{appendix1}).
 }
\end{remark}

\begin{proof}
Fix $x\in\Gw$ and set $u:=v_0(x)$. By Lemma~\ref{weak_lapl} we have
\begin{multline}\label{Lapl_eq}
Q_{V_1}(v_1)=-\Delta_p\big(\varphi_1(v_0)\big)+V_1\left(\varphi_1(v_0) \right)^{p-1}=\\
|\varphi_1'(u)|^{p\!-\!2}|\nabla v_0|^p\!\!\left(\!\!-(p\!-\!1)\varphi_1''(u)\!-\!\frac{\Delta_p (v_0)}{|\nabla v_0|^p}\varphi_1'(u)\!+\!V_1\frac{|\!\left(\!\log \varphi_1\!\right)'(u)|^{2\!-\!p}}{|\nabla v_0|^p}\varphi_1(u)\!\right)
\end{multline}
in the weak sense. On the other hand, with the identity map $\varphi_0(t):=t$ on $\R_+$ we have at $x$
\begin{multline}\label{Lapl_eq6}
Q_{V_0}(v_0)=-\Delta_p(\varphi_0(v_0)\big)+V_0\left(\varphi_0(v_0)\right)^{p-1}=\\
|\varphi_0'(u)|^{p\!-\!2}|\nabla v_0|^p\!\left(\!\!-(p\!-\!1)\varphi_0''(u)\!-\!\frac{\Delta_p (v_0)}{|\nabla v_0|^p}\varphi_0'(u)\!+\!V_0\frac{|\left(\!\log \varphi_0\!\right)'(u)|^{2\!-\!p}}{|\nabla v_0|^p}\varphi_0(u)\!\!\right).
\end{multline}

Therefore, for $j=0,1$, $\varphi_j(u)$ satisfies at the point $u$ the following linear ordinary differential inequality
$$-(p-1)\varphi_j''(u)-\frac{\Delta_p (v_0)}{|\nabla v_0|^p}\varphi_j'(u)+V_j\frac{|\left(\log \varphi_j\right)'(u)|^{2-p}}{|\nabla v_0|^p}\varphi_j(u)\; \displaystyle{\substack{= \\[1mm] (\geq)}}\; 0.$$
Denote $\varphi_\alpha(u):=\varphi_0(u)^{1-\alpha} \varphi_1(u)^{\alpha}$, and apply the one-dimensional version of Corollary~\ref{cor_conv}. We obtain the following linear differential inequality at $u$
\begin{multline}\label{mult1}
-(p-1)\varphi_\alpha''(u)-\frac{\Delta_p (v_0)}{|\nabla v_0|^p}\varphi_\alpha'(u)+(1-\alpha) V_0\frac{|\left(\log \varphi_0\right)'(u)|^{2-p}}{|\nabla v_0|^p}\varphi_\alpha(u)+\\
\alpha V_1\frac{\left|\!\left(\log \varphi_1\right)'(u)\right|^{2-p}}{|\nabla v_0|^p}\varphi_\alpha(u)\!-\!(p\!-\!1)\alpha(1\!-\!\alpha)\left|\left[\log \left(\frac{\varphi_0(u)}{\varphi_1(u)}\right)\right]'\right|^2\!\!\varphi_\alpha(u)
\; \displaystyle{\substack{= \\[1mm] (\geq)}}\; 0.
\end{multline}
In view of Lemma~\ref{weak_lapl} we have
$$-\Delta_p(\varphi_\ga)=
|\varphi_\ga'|^{p-2}|\nabla v_0|^p\left(-(p-1)\varphi_\ga''\!-\!\frac{\Delta_p (v_0)}{|\nabla v_0|^p}\varphi_\ga'\right).$$
On the other hand,
\begin{align*}
&|\left(\log \varphi_j\right)'|^{2-p}|\left(\log \varphi_\ga\right)'|^{p-2}=\left|\nabla \log v_j\right|^{2-p}\left|\nabla \log v_\ga\right|^{p-2}\quad j=0,1,\\[4mm]
& \left|\left[\log \left(\frac{\varphi_0}{\varphi_1}\right)\right]'\right|^2|\left(\log \varphi_\ga\right)'|^{p-2}|\nabla v_0|^p=\left|\nabla \log \left(\frac{v_0}{v_1}\right)\right|^2  \left|\nabla \log v_\ga\right|^{p-2}.
\end{align*}
Hence, \eqref{mult1} implies the result of the theorem.

\end{proof}
\begin{remark}{\em
In particular, let $V=0$ and  $v_0=G$ be the $p$-Laplacian's Green function with a pole at $0\in \Omega$, and $v_1=\mathbf{1}$. Then $V_\ga=0$, and a computation shows that $W_\ga=(p-1)\alpha^{p-1}(1-\alpha)\left|\frac{\nabla G}{G}\right|^p$ (cf. Proposition~\ref{supersolution p-Laplacian}).
 }
\end{remark}
\begin{corollary}\label{cor1-G}
Assume that $p>n$, $V=0$, and $-\Gd_p$ is subcritical in $\Gw$. Let $G$ be (up to a constant) the $p$-Green function with a pole at $0\in \Omega$. Suppose that
$$\lim_{x\to 0}G(x)=\gg >0 \quad \mbox{and } \;\lim_{x\to \bar{\infty}}G(x)=0.$$  For $0\leq \ga\leq 1$, let
\begin{equation}\label{eqW1-G}
v_\alpha\!:=\!G^{1\!-\!\alpha}(\gg\!-\!G)^\alpha, \;\;     W_\alpha:=\alpha(1\!-\!\alpha)(p\!-\!1)\big|\gg(1\!-\!\ga)\!-\!G\big|^{p\!-\!2}\left|\frac{\nabla G}{G(\gg\!-\!G)}\right|^p
\end{equation}
Then the following improved Hardy inequality holds in $\Omega^\star$:
\begin{equation}\label{eq_hardy_g_1-g}
\int_{\Omega^\star} |\nabla \vgf|^p\dnu\geq \int_{\Omega^\star} W_\ga |\vgf|^{p}\dnu \qquad \forall \vgf\in C_0^\infty(\Omega^\star).
\end{equation}
Moreover, for any $0\leq \ga\leq 1$ the operator $Q_{-W_\ga}$ is subcritical in $\Gw$.
\end{corollary}

\begin{proof}
By our assumption, $\gg-G$ is a positive $p$-harmonic function in $\Omega^\star$. Apply Theorem \ref{supersol_cons} with $v_0=G$ and $v_1=\gg-G$ to obtain \eqref{eq_hardy_g_1-g}.

Assume to the contrary that $Q_{-W_\alpha}$ is critical in $\Gw^\star$. Two cases should be considered: either $\alpha<(p-1)/p$, or $1-\alpha<(p-1)/p$.

Let us assume for example $\alpha<(p-1)/p$, the other case being similar (exchanging the roles of zero and infinity). Then $v_{\frac{p-1}{p}}$ is a positive supersolution of $Q_{-W_\alpha}$ in a neighborhood of zero, and $v_\alpha$ is a positive solution of $Q_{-W_\alpha}$ of minimal growth in $\Omega^\star$. Therefore, there exists $C>0$ such that $v_\alpha\leq Cv_{\frac{p-1}{p}}$ in a neighborhood of zero. But since $\alpha<(p-1)/p$, this is impossible, and we get a contradiction.
\end{proof}
\begin{remark}
{\em
A priori it is clear that for $W_\ga$ (given by \eqref{eqW1-G}) to be optimal at the origin, it is needed that $\alpha=(p-1)/p$, but for the constant to be optimal at $\bar{\infty}$, we must choose $\ga=1/p$, and thus $v_\ga$ cannot  be a ground state (if $p\neq 2$). Thus, in the nontrivial cases ($v_j\neq \mathrm{constant}$), the supersolution construction of the form $v_\ga=v_1^{\ga}v_0^{1-\ga}$, does not provide us with an optimal Hardy weight. On the other hand, let $\psi(G):=[G(\gg-G)]^{(p-1)/p}$ and
\begin{multline}\label{WG_1-G}
W:= \frac{-\Gd_p(\psi(G))}{\psi(G)^{p-1}}=\\[2mm]
 \left(\frac{p-1}{p}\right)^p \left|\frac{\nabla G}{G(\gg-G)}\right|^p
|\gg-2G|^{p-2}\big[ 2(p-2)G(\gg-G)+\gg^2 \big]
\geq 0.
\end{multline}
Then under the conditions of Theorem~\ref{thm_opt_hardy}, $W$ is an optimal Hardy-weight for $-\Gd_p$, and $\psi(G)$ is the ground state of the critical operator $Q_{-W}$ in $\Gw^\star$. Note that nevertheless, $W=0$ on the set $\{x\in\Gw^\star\mid G(x)=\gg/2\}$.
 }
\end{remark}
It turns out that if $V_j$ both have the same definite sign, then one can find potentials $\mathcal{V}_\ga\geq V_\ga$ (with the same definite sign) which does not depend on $v_j$, such that the corresponding Hardy inequality is satisfied with the same Hardy-weight $W_\ga$. We have
\begin{corollary}\label{thm_super_p_constr}
Let $\Gw$, $V_j$, $v_j$ (where  $j=0,1$), $v_\ga$, and $W_\ga$ be as in Theorem~\ref{supersol_cons} (or as in Theorem~\ref{thm_super_p_constr_o}).
 Suppose further that $V_j\geq 0$ if $1<p \leq 2$ (resp., $V_j\leq 0$ if $p \geq 2$), where $j=0,1$.   Define
 $$\mathcal{V}_\ga:=\pm \Big((1-\ga)|V_0|^{1/(p-1)}+\ga |V_1|^{1/(p-1)} \Big)^{p-1},$$
where one should take the minus sign if $V_j\leq 0$.
Then $v_\ga$ is a positive (super)solution of the equation
\begin{equation}\label{eqQv-Wa2}
  Q_{\mathcal{V}_\ga-W_\ga}(u)= 0 \qquad \mbox{ in } \Gw,
\end{equation}
and the following improved inequality holds
$$\mathcal{Q}_{\mathcal{V}_\ga}(\vgf) \geq \int_\Gw W_\ga |\vgf|^{p}\dnu \qquad \forall \vgf\in C_0^\infty(\Omega). $$
Moreover, if $p\neq 2$, and $|V_0|+|V_1|\neq0$, then the functional $\mathcal{Q}_{\mathcal{V}_\ga-W_\ga}$ is {\em subcritical} in $\Omega$.
\end{corollary}
\begin{proof}
Assume that the conditions of Theorem~\ref{supersol_cons} are satisfied. Then $v_\ga$ is a positive (super)solution of the equation $Q_{V_\ga-W_\ga}(u)=0$.

We claim that the function $(\xi,\eta)\mapsto f(\xi,\eta):=\xi^{p-1}\eta^{2-p}$ on $\mathbb{R}^2_+$ is convex (resp., concave) if $p\geq 2$ (resp., $p\leq 2$). Indeed,
$$\mathrm{Hess}\,(f)=(p-1)(p-2)\xi^{p-1}\eta^{2-p}\left[\begin{array}{cc}
\frac{1}{\xi^2} \quad& -\frac{1}{\xi\eta} \\[4mm]
-\frac{1}{\xi\eta} \quad& \frac{1}{\eta^2}
  \end{array}\right]\,,
$$
and it can be easily checked that $\mathrm{Hess}\,(f)$ is nonnegative (resp., nonpositive) on $\mathbb{R}^2_+$ if and only if $(p-1)(p-2)\geq 0$
(resp.,  $(p-1)(p-2)\leq 0$). Hence,
\begin{multline*}
\left[(1-\ga)|V_0|^{\frac{p-1}{p-1}}\left|\nabla\log v_0\right|^{2-p} +\ga |V_1|^{\frac{p-1}{p-1}}\left|\nabla\log v_1\right|^{2-p}\right]
\; \displaystyle{\substack{\geq \\[1mm] \big(\mbox{respect. }\leq\big)}}\;\\
 \Big((1-\ga)|V_0|^{1/(p-1)}+\ga |V_1|^{1/(p-1)} \Big)^{p-1}  \Big|(1-\ga)\nabla\log v_0 +\ga \nabla\log v_1\Big|^{2-p}\!\! .
\end{multline*}

So, $\mathcal{V}_\ga\geq V_\ga$, and hence $v_\ga$ is a positive supersolution of the equation
$$Q_{\mathcal{V}_\ga-W_\ga}(u)= 0\qquad \mbox{in }  \Omega,$$
and we have
$$\mathcal{Q}_{\mathcal{V}_\ga}(\vgf) \geq \int_\Gw W_\ga |\vgf|^{p}\dnu \qquad \forall \vgf\in C_0^\infty(\Omega). $$

If and $|V_0|+|V_1|\neq 0$, and $p\neq 2$, then the strict convexity (resp., concavity) of $f$ implies that $v_\ga$ is a positive supersolution of $Q_{\mathcal{V}_\ga-W_\ga}(u)=0$ which is not a solution, and therefore by Lemma~\ref{lem_crit}, the corresponding improved functional $\mathcal{Q}_{\mathcal{V}_\ga-W_\ga}$ is subcritical in  $\Omega$.

\end{proof}
\begin{remark}\label{rem4}
{\em
1.   Suppose that $V_0= V_1\neq 0$ and $V_0$ has a definite sign, then $\mathcal{V}_\ga=V_0$. By Corollary~\ref{thm_super_p_constr},
the operator $Q_{V_0-W_\ga}$ is {\em subcritical} in $\Omega$ if $p\neq 2$. This is in contrast with the linear case where $p=2$. Indeed, if $v_0$ is the Green function of the operator $Pu:=-\div\big(A(x)\nabla \cdot\big)+V(x)$  in $\Gw$ with a pole $0$, and if $v_1$ is a positive solution satisfying
$\lim_{x\to \bar{\infty}} \frac{v_0(x)}{v_1(x)}=0$, then $P-W_{1/2}=P-
 \frac{1}{4}\left|\nabla \log\left(\frac{v_0}{v_1}\right)\right|^2$ is critical in $\Gw^\star$ (see \cite[Theorem~2.2]{DFP}).

2. In general, it is not clear how to optimize in $\ga$ the potentials $W_\ga$ in the case $V_0= V_1\neq 0$, and $V_0$ has a definite sign (so, $\mathcal{V}_\ga=V_0$). But if we take $v_0=\mathbf{1}$ (so, $V\geq 0$ and $1<p \leq 2$), and $v_1=v$ is a positive supersolution of the equation $Q_{V_0}(u)=0$, then $$W_\ga= \ga^{p-1}(1-\ga)(p-1)\left|\frac{\nabla v}{v}\right|^{p},$$
and by optimizing $\ga$ one obtains
\begin{equation}\label{eqhardyrad}
\mathcal{Q}_V(\vgf) \geq \left(\frac{p-1}{p}\right)^p\int_{\Omega} \left(\frac{|\nabla v|}{v}\right)^p |\vgf|^{p}\dnu \quad \forall \vgf\in C_0^\infty(\Omega),
\end{equation}
which in particular reproves (2.12) in \cite{AS} if $A$ is the identity matrix.
 }
\end{remark}
\appendix
\section{Radially symmetric potentials}\label{appendix1}
In this Appendix we present a proof of a particular case of Theorem~\ref{supersol_cons}, where the two positive supersolutions are radially symmetric functions, and in particular, have the same level sets.
\begin{theorem}\label{thm_super_p_constr_o}
Assume that for $j=0,1$
\be\label{Q_V2}
\mathcal{Q}_{V_j}(\vgf):=\int_\Omega(|\nabla \vgf|^p+V_j|\vgf|^p)\dnu\geq 0\qquad \vgf\in C_0^\infty(\Omega),
\ee
where $\Omega$ is a domain in $\mathbb{R}^n$ not containing the origin, and  the potentials $V_j$ are two radially symmetric potentials.  Let $v_j$,  $j=0,1$, be two positive, linearly independent, radially symmetric, $C^2$-supersolutions of the equation $Q_{V_j}(u)=0$ in $\Omega$.  For  $0\leq \ga\leq 1$, define the function $$v_\ga(r):= (v_1(r))^{\ga}(v_0(r))^{1-\ga},$$
where $r:=|x|$. Assume further that $(v_0)'$, $(v_1)'$, and $(v_\ga)'$ do not vanish, and let
 $$V_\ga(r)\!:=\!\Big(\!(1-\ga)V_0(r)|(\log v_0(r))'|^{2\!-\!p}\!+\!\ga V_1(r)|(\log v_1(r))'|^{2\!-\!p} \!\Big)\!|(\log v_\ga(r))'|^{p\!-\!2}\!\!,$$
 $$W_\ga(r):= \ga(1-\ga)(p-1)
 \left|\left[\log\left(\frac{v_0(r)}{v_1(r)}\right)\right]'\right|^2|(\log v_\ga(r))'|^{p-2}.$$
Then $v_\ga$ is a positive supersolution of the equation
\begin{equation}\label{eqQv-Wa2_o}
  Q_{V_\ga(|x|)-W_\ga(|x|)}(u)= 0 \qquad \mbox{ in } \Gw,
\end{equation}
and the following improved inequality holds
$$\mathcal{Q}_{V_\ga}(\vgf) \geq \int_\Gw W_\ga |\vgf|^{p}\dnu \qquad \forall \vgf\in C_0^\infty(\Omega). $$
\end{theorem}
\begin{proof}
Assume that $v$ is a radially symmetric $C^2$-function, and denote $r:=|x|$, $v':=\mathrm{d}v/\mathrm{d}r$. Then by Lemma~\ref{weak_lapl} the  $p\,$-Laplacian of $v$ satisfies
\begin{equation}\label{eq:green_o}
-\pl (v) =-\frac{1}{r^{n-1}}\left( r^{n-1}|v'|^{p-2} v'  \right)' =
-|v'|^{p-2}\left[(p-1)v''+\frac{n-1}{r}v'\right]
\end{equation}
in the weak sense. Denote the linear operator
$$Pu:=-(p-1)u''-\frac{n-1}{r}u'.$$
By our assumptions, $v_j$ are positive radial (super)solutions of the equation $Q_{V_j}(u)=0$ in $\Omega$, where $j=0,1$.
Hence,
$$Pv_j  +\Big(V_j\left|(\log v_j)'\right|^{2-p}\Big) v_j\; \displaystyle{\substack{= \\[1mm] (\geq)}}\; 0\qquad j=0,1.$$

  Therefore, by Lemma~\ref{lem_conv}, $v_\ga$ is a positive (super)solution of the linear equation
\begin{multline}
\left[P+(1-\ga)V_0\left|(\log v_0)'\right|^{2-p}+\ga V_1\left|(\log v_1)'\right|^{2-p} -\right.\\[2mm]
\left. (p-1)\alpha(1-\alpha) \left|\left[\log\left(\frac{v_0}{v_1}\right)\right]'\right|^2 \right]u\; \displaystyle{\substack{= \\[1mm] (\geq)}}\;0.
\end{multline}
Hence, $v_\ga$ satisfies the quasilinear differential (in)equality
\begin{multline}
-\Gd_p(v_\ga)\!+\!\left(\!(1-\ga)V_0\left|(\log v_0)'\right|^{2\!-\!p}\!\!+\!\ga V_1\left|(\log v_1)'\right|^{2\!-\!p}\right)\!\left|\left(\log v_\ga\right)'\right|^{p\!-\!2} \!v_\ga^{p\!-\!1} -\\[2mm]
(p\!-\!1)\alpha(1\!-\!\alpha) \left|\left[\log\left(\frac{v_0}{v_1}\right)\right]'\right|^2\!\! \left|\left(\log v_\ga\right)'\right|^{p-2}v_\ga^{p-1} = Q_{V_\ga-W_\ga}(v_\ga)\; \displaystyle{\substack{= \\[1mm] (\geq)}}\; 0.
\end{multline}
\end{proof}
\begin{remark}\label{rem11}
{\em
If $0\in \Gw$, $\Gw$ is  a radially symmetric domain, $V_0=V_1$ is a radially symmetric potential, and  $Q_{V_0}$ is subcritical in $\Gw$, then one can apply Theorem~\ref{thm_super_p_constr_o} in $\Gw^\star=\Gw\setminus \{0\}$  with $v_0$ equals to the corresponding (unique) $p\,$-Green function of $Q_{V_0}$ with a pole at the origin, and $v_1$  a global radial positive supersolution of the equation $Q_{V_0}(u)=0$ in $\Omega$.
 }
\end{remark}
\begin{center}{\bf Acknowledgments} \end{center}
The authors wish to thank Martin Fraas for valuable discussions.
They acknowledge the support of the Israel Science
Foundation (grants No. 963/11) founded by the Israel Academy of
Sciences and Humanities. B.~D. was supported in part by a Technion fellowship.

\end{document}